\documentclass[12pt]{amsart}
 \usepackage[margin=1.4in]{geometry}
 \usepackage[a4paper, left=1.3in, right=.3in]{}
 \usepackage{amsmath,amssymb,amsthm,graphicx,amsxtra,setspace}
 \usepackage[utf8]{inputenc}
 \usepackage{mathrsfs}
 \usepackage{hyperref}
 \usepackage{xcolor}
 \usepackage{upgreek}
 \usepackage{mathtools}
 \allowdisplaybreaks
 
 \newtheorem{theorem}{Theorem}[section]

 \newtheorem{definition}{Definition}[section]

 \usepackage{latexsym}
 \usepackage{hyperref}
 \usepackage{graphicx}
 \usepackage{epstopdf}
\usepackage{bigints}
\usepackage{siunitx}

\usepackage{amssymb,bbm,enumerate,bbm,amsmath}
\usepackage{changepage}

\usepackage{mathrsfs}
\usepackage{tikz}
\usepackage{ dsfont}
\usepackage{hyperref}
\usetikzlibrary{arrows}
\usepackage{amsthm}
\usepackage{array,amsfonts}
\usepackage{amsmath}
\usepackage[utf8]{inputenc}
\usepackage[T1]{fontenc}
\usepackage{mathtools}
\usepackage{enumitem}
\usepackage[thinc]{esdiff}
\usepackage{multirow}
\usepackage{graphics}
\usepackage{amsmath,bm}

\newtheorem{lemma}[theorem]{Lemma}
\newtheorem{corollary}[theorem]{Corollary}

\newtheorem{note}{Note}

 \allowdisplaybreaks
 
 \let\originalleft\left
 \let\originalright\right
 \renewcommand{\left}{\mathopen{}\mathclose\bgroup\originalleft}
 \renewcommand{\right}{\aftergroup\egroup\originalright}

 \newcommand{\Addresses}{{
 		\footnote{

 			\noindent	 \textsuperscript{1,2} Department of Mathematics, Indian Institute of Technology Roorkee, Roorkee, 247667, India.	
 			
 			\noindent  \textit{e-mail\textsuperscript{1}:} \texttt{p\_yadav@ma.iitr.ac.in}
 			
 			\noindent  \textit{e-mail\textsuperscript{2}:} \texttt{tanuja.srivastava@ma.iitr.ac.in}

 }}}
 
\begin{document}
 	\title[]{Maximum Likelihood Estimation of the Parameters of Multivariate and Matrix Variate Symmetric Laplace Distribution \Addresses}
 	\author [Pooja Yadav.  Tanuja Srivastava]{Pooja Yadav\textsuperscript{1}.  Tanuja Srivastava\textsuperscript{2}}

\begin{abstract}
	This paper considers the matrix variate symmetric Laplace distribution, which is a scale mixture of the matrix normal distribution. In this paper, the maximum likelihood estimators (MLE) of the parameters of multivariate and matrix variate symmetric Laplace distributions are obtained with the help of the Expectation-Maximization (EM) algorithm. The parameters of the matrix variate symmetric Laplace distribution, along with their MLEs, are defined up to a positive multiplicative constant, and their Kronecker product is uniquely determined. The conditions for the existence of the MLEs are given. The performance of the obtained estimators are evaluated with respect to bias and mean Euclidean distance. In case of the multivariate symmetric Laplace distribution, the proposed estimator is also compared with another estimator given by \cite{EKL} and \cite{KS}. Furthermore, the empirical bias and the mean Euclidean distance of the Kronecker product for the estimators of the matrix variate symmetric Laplace distribution are analyzed using simulated data across different sample sizes. 
\end{abstract}
\maketitle

\section{Introduction}\label{sec1}  
	
	The Laplace distribution is a most helpful tool for modelling data that has sharp peaks at the location parameter and heavy tails, which are common in many real-world applications such as finance, biological sciences and engineering sciences, where the Laplace distribution provides better fits for the empirical data than the normal distribution (see \cite{FM}; \cite{KKP}; \cite{KP}; \cite{KPR}). The multivariate versions of the univariate Laplace distribution have been studied by many authors, all these versions are called the multivariate Laplace distribution. The term "\emph{multivariate Laplace law}" is now commonly used for symmetric or elliptically contoured distributions; these distributions possess the characteristic function depending on their variable through a quadratic form only (see \cite{CHS}). 
	
	The multivariate symmetric Laplace distribution is a specific case of the multivariate asymmetric Laplace distribution, where the location parameter is always assumed to be zero. The multivariate asymmetric Laplace distribution for modelling the skewed data is proposed by \cite{KPo}. The multivariate skew Laplace distribution, which is an alternative to the multivariate asymmetric Laplace distribution, was introduced by \cite{A}, using the variance-mean mixture of the multivariate normal distribution and the inverse gamma distribution. For more details about these two distributions, the readers are referred to \cite{A} and \cite{KKP}.
	
	\cite{BA} introduced the matrix variate skew Laplace distribution, which is a generalization of the multivariate skew Laplace distribution given by Arslan, using mean-variance mixtures of the normal distribution. They found the MLEs of the parameters of the matrix variate skew Laplace distribution using the EM algorithm and did a simulation study. 
	
	The matrix variate asymmetric Laplace distribution and matrix variate generalized asymmetric Laplace distribution are introduced by \cite{Y}, which are the extensions of multivariate asymmetric Laplace and multivariate generalized asymmetric Laplace distributions to the matrix variate case. Later, \cite{KMP} studied the matrix variate generalized asymmetric Laplace distribution using the univariate scale mixture of the matrix normal distribution and the matrix scale mixture of the matrix normal distribution. 
	
	The maximum likelihood estimators (MLEs) of the parameters of the multivariate asymmetric Laplace distribution are studied in \cite{EKL}. However, an estimate of the scale parameter is derived by taking the scale parameter as a diagonalizable matrix using the EM algorithm. Estimation of the parameters of the multivariate Laplace distribution is studied by using the other method in \cite{EKL}, \cite{KS} and \cite{Vi}.
	
	This paper considers an equivalent definition of the matrix variate symmetric Laplace distribution using the vectorization of the random matrix. In this paper, we explore the maximum likelihood estimation of the parameters of the matrix variate symmetric Laplace distribution. Deriving the closed-form expressions for the MLE of the parameters is not straightforward due to the presence of the modified Bessel function of the third kind in the probability density function. So, we employ the Expectation-Maximization (EM) algorithm to estimate the parameters and propose a simple iterative algorithm to compute the MLEs. Additionally, we present the EM algorithm for the maximum likelihood estimator of the scale parameter of the multivariate symmetric Laplace distribution. We also establish the necessary conditions for the existence of the MLEs for the parameters of both multivariate and matrix variate symmetric Laplace distributions. To demonstrate the performance of the proposed algorithm, we compared the EM estimator with another estimator of the multivariate symmetric Laplace distribution by evaluating the bias and mean Euclidean distance of these estimators. Additionally, we simulate the empirical bias and the mean Euclidean distance of the Kronecker product of estimators of the matrix variate symmetric Laplace distribution.
	
	This work may be a valuable addition to the application, where matrix variate and multivariate symmetric Laplace distributions are suitable probabilistic tools. One of the most direct applications of the matrix variate symmetric Laplace distribution is in panel data. This data is commonly used in economics and finance (see \cite{KMP}). In the case of limited data availability, the matrix variate symmetric Laplace distribution can be used in place of the multivariate symmetric Laplace distribution, where the scale parameter matrix is a Kronecker product of two positive definite matrices.

	The paper is organized as follows, in section $2$, an alternative definition of the matrix variate symmetric Laplace distribution is given and from this definition, the probability density function is derived and some preliminary results such as, the characteristic function and representation of the multivariate and matrix variate symmetric Laplace distributions are given. In section $3$, the MLE of parameters using the EM algorithm is proposed, which is in the form of a simple iterative algorithm. In section $4$, a necessary and sufficient condition for the existence of MLE of the parameters for both multivariate and matrix variate symmetric Laplace distribution are established. In section $5$, the performance of the proposed EM estimator is compared with another estimator of the multivariate symmetric Laplace distribution by evaluating bias and mean Euclidean distance of these estimators and, the empirical bias and the mean Euclidean distance of the Kronecker product of the estimators of the matrix variate symmetric Laplace distribution are shown using the simulation. Section $6$ contains the conclusion of the paper.

	\section{Preliminaries}
	\noindent\textbf{\emph{Notations}.} The following notations are used throughout the paper: $\mathcal{N}_{p}(\bm{0},\bm{\Sigma})$ denotes the $p$-dimensional multivariate normal distribution where $\bm{0}$ is a $p$-dimensional vector with zero entries, and $\bm{\Sigma}$ is a $p \times p$ positive definite matrix. $\operatorname{tr}(\bm{A})$ and $\begin{vmatrix} \bm{A}	\end{vmatrix}$ denotes the trace and  the determinant of the matrix $\bm{A}$, respectively. If $\bm{A}$ is a matrix, then $\operatorname{diag}(\bm{A})$ is a diagonal matrix with only diagonal elements of $\bm{A}$. If $\bm{A}$ is a matrix of order $m\times n$, then the column-wise vectorization of matrix $\bm{A}$ of order $mn\times 1$ is denoted as $\emph{vec}(\bm{A})$. $\bm{A}\otimes \bm{B}$ denotes the Kronecker product of matrices $\bm{A}$ and $\bm{B}$. $\bm{A}^\top$ denotes the transpose of the matrix $\bm{A}$. The notation $\mathcal{MN}_{p,q}(\bm{0}, \bm{\Sigma}_{1}, \bm{\Sigma}_{2})$ is used for matrix variate normal distribution, where $\bm{0}$ is a matrix of order $p\times q$ with all entries zero and $\bm{\Sigma}_{1}, \bm{\Sigma}_{2}$ are positive definite matrices of order $p \times p$ and $q\times q$, respectively. The notation ${Exp}(1)$ is used for the univariate exponential distribution with location parameter $0$ and scale parameter $1$. The notations $\mathcal{SL}_{p}(\bm{\Sigma})$ and $\mathcal{MSL}_{p,q}(\bm{\Sigma}_{1},\bm{\Sigma}_{2})$ are used for $p$-dimensional multivariate symmetric Laplace distribution and matrix variate symmetric Laplace distribution, respectively. $\|.\|_{2}$ denotes Euclidean norm or Frobenius norm of matrices. 
	
	\subsection{Multivariate symmetric Laplace distribution}
	The density function of a $p$-dimensional symmetric Laplace distributed random vector $ \bm{Y}= (y_{1},y_{2},\cdots,y_{p})^{\top}$, $\bm{Y} \in \mathbb{R}^p$ with location parameter zero and scale parameter $\bm{\Sigma}_ {p\times p}$ (positive definite matrix), is
	\begin{equation}
		\label{eq:1}
		f_{\bm{Y}}(\bm{y}) = \frac{2}{(2\pi)^{\frac{p}{2}} \begin{vmatrix} \bm{\Sigma} \end{vmatrix}^{\frac{1}{2}}} \left( \frac{\bm{y}^\top \bm{\Sigma}^{-1} \bm{y}}{2} \right) ^{\nu/2}  K_{\nu} \left( \sqrt{2\bm{y}^\top \bm{\Sigma}^{-1} \bm{y}} \right),
	\end{equation}
	here, $\nu = \frac{2-p}{2}$, and $ K_{\nu}$ is the modified Bessel function of the third kind. This distribution is denoted as $\bm{Y} \sim \mathcal{SL}_{p}(\bm{\Sigma})$ (see \cite{KKP}). The readers are referred to \cite{B}, \cite{OLBC}, \cite{W} for the definition and properties of the modified Bessel function of the third kind. 
	
	\begin{theorem}
		\label{thm:1}
		 A multivariate symmetric Laplace random variable $\bm{Y}$ has the representation 
		\begin{equation}
			\label{eq:2} 		
			\bm{Y}=\sqrt{W}\bm{Z},
		\end{equation}
		with random variable $\bm{Z}\sim \mathcal{N}_{p}(\bm{0}, \bm{\Sigma})$, the $p$-dimensional normal distribution with location parameter zero and scale parameter $\bm{\Sigma}$ and random variable $W$, independent of $\bm{Z}$, having a univariate exponential distribution with location parameter zero and scale parameter one.
	\end{theorem}
    \begin{proof}
    	Since $W$ and $\bm{Z}$ are independent and $\bm{Y}=\sqrt{W}\bm{Z}$, then, the joint probability density function of $\bm{Y}$ and $W$ is
   	\begin{align*}
   		f_{\bm{Y},W}(\bm{y},w)&= f_{\bm{Z},W}(\bm{z},w) \left| J \right| \\
   		&= f_{\bm{Z}}(\bm{z}) f_{W}(w) \left| J \right|,
   	\end{align*}
   	where $J = \begin{vmatrix} \frac{\partial(\bm{Z},W)}{\partial(\bm{Y},W)}	\end{vmatrix} =\frac{1}{(W)^{\frac{p}{2}}}$ is the jacobian. Hence, the joint density function of $\bm{Y}$ and $W$ is
   	
   	\begin{equation}
   		\label{eq:1.1}
   		f_{\bm{Y},W}(\bm{y},w)= \frac{1}{(2\pi)^{\frac{p}{2}} \begin{vmatrix} \bm{\Sigma} \end{vmatrix}^{\frac{1}{2}}  w^\frac{p}{2}} \exp\left(-w-\frac{1}{2w} \bm{y}^\top  \bm{\Sigma}^{-1} \bm{y}\right),
   	\end{equation}
   and the density function of $\bm{Y}$ is
   
   \begin{align*}
   		f_{\bm{Y}}(\bm{y}) &= \int_{0}^{\infty} \frac{1}{(2\pi)^{\frac{p}{2}} \begin{vmatrix} \bm{\Sigma} \end{vmatrix}^{\frac{1}{2}}  w^\frac{p}{2}} \exp\left(-w-\frac{1}{2w} \bm{y}^\top  \bm{\Sigma}^{-1} \bm{y}\right) \,dw\\
   		                                &=\frac{1}{(2\pi)^{\frac{p}{2}} \begin{vmatrix} \bm{\Sigma} \end{vmatrix}^{\frac{1}{2}} } \bigintsss_{0}^{\infty} \frac{1}{\left(w^{\frac{p-2}{2}+1}\right)}  \exp\left(-w-\frac{\left(\sqrt{2\bm{y}^\top  \bm{\Sigma}^{-1} \bm{y}}\right)^{2}}{4w} \right) \,dw \\
   		                                &= \frac{2}{(2\pi)^{\frac{p}{2}} \begin{vmatrix} \bm{\Sigma} \end{vmatrix}^{\frac{1}{2}} } \left(\frac{\bm{y}^\top  \bm{\Sigma}^{-1} \bm{y}}{2}\right)^{\nu/2}  K_{-\nu}\left(\sqrt{2\bm{y}^\top  \bm{\Sigma}^{-1} \bm{y}}\right),
   \end{align*}
  where $\nu=\frac{2-p}{2}$, and $K_{-\nu}(x)$ is the modified Bessel function of the third kind given as
  
  \begin{equation*}
  	K_{-\nu}(x)=\frac{1}{2}\left(\frac{x}{2}\right)^{-\nu} \int_{0}^{\infty} \frac{1}{(t)^{-\nu+1}} \exp\left(-t-\frac{x^2}{4t}\right) \,dt .
  \end{equation*}

  From the properties of modified Bessel function of the third kind, $K_{-\nu}(x)=K_{\nu}(x)$. Hence, the density function of $\bm{Y}$ is
\begin{equation*}
	f_{\bm{Y}}(\bm{y})=\frac{2}{(2\pi)^{\frac{p}{2}} \begin{vmatrix} \bm{\Sigma} \end{vmatrix}^{\frac{1}{2}} } \left(\frac{\bm{y}^\top  \bm{\Sigma}^{-1} \bm{y}}{2}\right)^{\nu/2}  K_{\nu}\left(\sqrt{2\bm{y}^\top  \bm{\Sigma}^{-1} \bm{y}}\right).
\end{equation*}
   \end{proof}
	
	\begin{lemma}
   \label{lemma:1}
        If the random vector $\bm{Y}$ and random variable $W$ have a joint density function $f_{\bm{Y},W}(\bm{y},w)$ given in equation \eqref{eq:1.1}, then the conditional expectation of $\frac{1}{W}$ given $\bm{Y}$ is 
   		\begin{equation*}
   				E\left(\frac{1}{W} \mid \bm{Y}=\bm{y} \right) = \left( \frac{\bm{y}^\top \bm{\Sigma}^{-1} \bm{y}}{2} \right) ^{-\frac{1}{2}} \frac{K_{\nu-1} \left( \sqrt{2 \left({\bm{y}^\top \bm{\Sigma}^{-1} \bm{y}} \right)}  \right) }{K_{\nu} \left( \sqrt{2 \left({\bm{y}}^\top \bm{\Sigma}^{-1} \bm{y}\right)}  \right) }.
   		\end{equation*}
 	\end{lemma}
 
   \begin{proof}
   		The conditional distribution of $W$ given $\bm{Y}$ is required to find the conditional expectation. The density function of the conditional distribution of $W$ given $\bm{Y}$ is 
   	
   	\begin{align*}   
   		f_{W|\bm{Y}}(w)&=  \frac{f_{\bm{Y},W}(\bm{y},w) }{f_{\bm{Y}}(\bm{y})}\\
   		&=\frac{\exp\left(-w-\frac{1}{2w}\bm{y}^\top \bm{\Sigma}^{-1}\bm{y} \right)}{2 w^{\frac{p}{2}} K_{\nu} \left( \sqrt{2 ({\bm{y}}^\top \bm{\Sigma}^{-1} \bm{y})}  \right) } \left( \frac{\bm{y}^\top \bm{\Sigma}^{-1} \bm{y}}{2} \right) ^{-\nu/2}.	
   	\end{align*}
   	
   	From this density function, the conditional expectation is
   	\begin{align*}
   		E\left(\frac{1}{W} \mid \bm{Y} \right) &= \int_{0}^{\infty} \frac{1}{w} \, f_{W|\bm{Y}} \left(w \right) \, dw \\
   		&=  \left( \frac{\bm{y}^\top \bm{\Sigma}^{-1} \bm{y}}{2} \right) ^{-\frac{1}{2}} \frac{K_{\nu-1} \left( \sqrt{2 \left({\bm{y}^\top \bm{\Sigma}^{-1} \bm{y}} \right)}  \right) }{K_{\nu} \left( \sqrt{2 \left({\bm{y}}^\top \bm{\Sigma}^{-1} \bm{y}\right)}  \right) }.
   	\end{align*}
   \end{proof}

	\subsection{\textbf{Matrix variate symmetric Laplace distribution: definition and properties}}
	
	The matrix variate symmetric Laplace distribution is an extension of multivariate symmetric Laplace distribution to matrix variate case. The matrix variate symmetric Laplace distribution is a special case of the matrix variate asymmetric Laplace distribution. The matrix variate asymmetric Laplace distribution is studied by \cite{KMP} and \cite{Y}.
	
	In this subsection, we reconsider the matrix variate symmetric Laplace distribution and its properties, such as the characteristic function and representation, as given by \cite{Y}. We provide an equivalent definition of the matrix variate symmetric Laplace distribution using the vectorization of the random matrix, similar to the approach taken for the matrix normal distribution discussed by \cite{GN}. Based on this definition, we have derived the probability density function of the matrix variate symmetric Laplace distribution. Also, this definition is useful in the multivariate setting, which will be addressed in further sections. 
	\begin{definition}[\textbf{Matrix variate symmetric Laplace distribution}]
		\label{def:1}
		A random matrix $\bm{X}$ of order $p\times q$ is said to have a matrix variate symmetric Laplace distribution with parameters $\bm{\Sigma}_{1}\in \mathbb{R}^{p\times p}$ and $\bm{\Sigma}_{2}\in \mathbb{R}^{q\times q}$ (positive definite matrices) if $vec(\bm{X})\sim \mathcal{SL}_{pq}(\bm{\Sigma}_{2} \otimes \bm{\Sigma}_{1})$. This distribution is denoted as $\bm{X} \sim \mathcal{MSL}_{p,q}(\bm{\Sigma}_{1},\bm{\Sigma}_{2})$.
	\end{definition} 
	Here, the term "symmetric" refers to the elliptically contoured distributions. For more details about matrix variate elliptically contoured distribution, see \cite{VA}.
	
	In the following theorem, the probability density function of the random matrix $\bm{X}\sim \mathcal{MSL}_{p,q}(\bm{\Sigma}_{1}, \bm{\Sigma}_{2})$ is derived.
	
	\begin{theorem}[\textbf{Probability density function}]
		If $\bm{X} \sim \mathcal{MSL}_{p,q}(\bm{\Sigma}_{1},\bm{\Sigma}_{2})$, then the probability density function of $\bm{X}$ is
		\begin{multline}
			\label{eq:3}
			f_{\bm{X}}(\bm{x})=\frac{2}{(2\pi)^\frac{pq}{2} \begin{vmatrix} \bm{\Sigma}_{2}\end{vmatrix}^{p/2} \begin{vmatrix} \bm{\Sigma}_{1}\end{vmatrix}^{q/2}} \left(\frac{\operatorname{tr}\left(\bm{\Sigma}_{2}^{-1} \bm{x}^\top \bm{\Sigma}_{1}^{-1} \bm{x}\right)} {2}\right)^{\frac{\nu}{2}} \\  K_{\nu} \left( \sqrt{ 2 \operatorname{tr}\left(\bm{\Sigma}_{2}^{-1} \bm{x}^\top \bm{\Sigma}_{1}^{-1} \bm{x}\right)}\right),
		\end{multline}
		where $\nu =\frac{2-pq}{2}$ and $K_{\nu}$ is the modified Bessel function of the third kind.
	\end{theorem}
	
	\begin{proof} From the definition \ref{def:1} and the probability density function given in \eqref{eq:1},
		$vec(\bm{X}) \sim \mathcal{SL}_{pq}(\bm{\Sigma}_{2} \otimes \bm{\Sigma}_{1})$ with probability density function  
		\begin{multline*}
			f_{vec(\bm{X})}(vec(\bm{x})) = \frac{2}{(2\pi)^{\frac{pq}{2}} \begin{vmatrix} \bm{\Sigma}_{2} \otimes \bm{\Sigma}_{1} \end{vmatrix}^{\frac{1}{2}}} \left( \frac{\left(vec(\bm{x})\right)^\top \left(\bm{\Sigma}_{2} \otimes \bm{\Sigma}_{1} \right)^{-1} vec(\bm{x})}{2} \right) ^{\frac{2-pq}{4}} \\
			K_{\frac{2-pq}{2}} \left( \sqrt{2\left(vec(\bm{x})\right)^\top \left(\bm{\Sigma}_{2} \otimes \bm{\Sigma}_{1}\right)^{-1} vec(\bm{x})} \right).
		\end{multline*}
		
		Using properties of Kronecker product, trace and determinants (see, \cite{GN}, \cite{RC}),
		\[\begin{vmatrix} \bm{\Sigma}_{2} \otimes \bm{\Sigma}_{1} \end{vmatrix} = \begin{vmatrix} \bm{\Sigma}_{2}\end{vmatrix}^{p} \begin{vmatrix} \bm{\Sigma}_{1}\end{vmatrix}^{q},\]
		
		and  
		\begin{align*}
			\left(vec(\bm{x})\right)^\top \left(\bm{\Sigma}_{2} \otimes \bm{\Sigma}_{1} \right)^{-1} vec(\bm{x})&=\left(vec(\bm{x})\right)^\top \left((\bm{\Sigma}_{2}^{-1})^\top \otimes \bm{\Sigma}_{1}^{-1}\right) vec(\bm{x}) \\
			&=\left(vec(\bm{x})\right)^\top vec\left( \bm{\Sigma}_{1}^{-1}  \bm{x} \bm{\Sigma}_{2}^{-1}\right) \\
			&=\operatorname{tr}\left(\bm{x}^\top \bm{\Sigma}_{1}^{-1}  \bm{x} \bm{\Sigma}_{2}^{-1}\right)\\
			&=\operatorname{tr}\left(\bm{\Sigma}_{2}^{-1} \bm{x}^\top \bm{\Sigma}_{1}^{-1} \bm{x}\right),
		\end{align*}
		
		therefore,
		\begin{equation*}
			\left(vec(\bm{x})\right)^\top \left(\bm{\Sigma}_{2} \otimes \bm{\Sigma}_{1} \right)^{-1} vec(\bm{x})	= \operatorname{tr}\left(\bm{\Sigma}_{2}^{-1} \bm{x}^\top \bm{\Sigma}_{1}^{-1} \bm{x}\right).
		\end{equation*}
		
		Hence, 
		\begin{multline*}
			f(\bm{x})=\frac{2}{(2\pi)^\frac{pq}{2} \begin{vmatrix} \bm{\Sigma}_{2}\end{vmatrix}^{p/2} \begin{vmatrix} \bm{\Sigma}_{1}\end{vmatrix}^{q/2}} \left(\frac{\operatorname{tr}\left(\bm{\Sigma}_{2}^{-1} \bm{x}^\top \bm{\Sigma}_{1}^{-1} \bm{x}\right)} {2}\right)^{\frac{2-pq}{4}}\\ K_{\frac{2-pq}{2}} \left( \sqrt{ 2 \operatorname{tr}\left(\bm{\Sigma}_{2}^{-1} \bm{x}^\top \bm{\Sigma}_{1}^{-1} \bm{x}\right)}\right).
		\end{multline*}
		
	\end{proof}

	\begin{theorem}[\textbf{Representation}]
		\label{theorem:3} If $\bm{Z}\sim \mathcal{MN}_{p,q}(\bm{0}, \bm{\Sigma}_{1}, \bm{\Sigma}_{2})$, $W\sim Exp(1)$ and $\bm{Z}$ and $W$ are independent. Then, the random matrix $\bm{X}= \sqrt{W} \bm{Z}$ has a matrix variate symmetric Laplace distribution with probability density function given in \eqref{eq:3}.
	\end{theorem}
  \begin{proof}
		Since by definition of matrix normal distribution (\cite{Du}),
		\begin{equation*}
			\bm{Z}\sim \mathcal{MN}_{p,q}(\bm{0},\bm{\Sigma}_{1},\bm{\Sigma}_{2}) \iff vec(\bm{Z}) \sim \mathcal{N}_{pq}(vec(\bm{0}),\bm{\Sigma}_{2} \otimes \bm{\Sigma}_{1}),
		\end{equation*}
	 and for $ \bm{X}=\sqrt{W} \bm{Z}$, \[vec(\bm{X})=vec(\sqrt{W} \bm{Z})=\sqrt{W} vec(\bm{Z}), \]\ 
	then, from the representation \eqref{eq:2} of multivariate symmetric Laplace distribution,  
	\[ vec(\bm{X})=\sqrt{W}vec(\bm{Z}) \sim \mathcal{SL}_{pq}(\bm{\Sigma}_{2} \otimes \bm{\Sigma}_{1}).\]
      Hence, by the definition \ref{def:1},
	\[\sqrt{W}\bm{Z}\sim \mathcal{MSL}_{p,q}(\bm{\Sigma}_{1}, \bm{\Sigma}_{2}).\]	
 \end{proof}

 \begin{corollary}
 	\label{corly:1}
 		If $\bm{X} \sim \mathcal{MSL}_{p,q}(\bm{\Sigma}_{1},\bm{\Sigma}_{2})$ and $W \sim Exp(1)$, then
 	 the joint probability density function of random matrix $\bm{X}$ and random variable $W$ is
 	 \[f_{\bm{X}, W}(\bm{x}, w)=\frac{\exp(-w)} {(2\pi)^\frac{pq}{2}(w)^\frac{pq}{2} \begin{vmatrix} \bm{\Sigma}_{2}\end{vmatrix}^{p/2} \begin{vmatrix} \bm{\Sigma}_{1}\end{vmatrix}^{q/2}} \exp\left(-\frac{1}{2w} \operatorname{tr}\left( \bm{\Sigma}_{2}^{-1} \bm{x}^\top \bm{\Sigma}_{1}^{-1} \bm{x} \right) \right). \]
 \end{corollary}
\begin{proof}
	From the above theorem \ref{theorem:3}, the representation of a matrix variate symmetric Laplace distributed random variable $\bm{X} = \sqrt{W} \bm{Z}$, then 
	\[vec(\bm{X})=\sqrt{W} vec(\bm{Z})\sim \mathcal{SL}_{pq}(\bm{\Sigma}_{2} \otimes \bm{\Sigma}_{1}).\] 
	
	From the theorem \ref{thm:1}, the joint probability density function of $vec(\bm{X})$ and $W$ is 
	
	\begin{equation*}
		f(vec(\bm{x}),w)= \frac{1}{(2\pi)^{\frac{pq}{2}} \begin{vmatrix} \bm{\Sigma}_{2} \otimes \bm{\Sigma}_{1} \end{vmatrix}^{\frac{1}{2}}  w^\frac{pq}{2}}\exp\left(-w-\frac{\left(vec(\bm{x})\right)^\top  \left(\bm{\Sigma}_{2} \otimes \bm{\Sigma}_{1}\right)^{-1} vec(\bm{x})}{2w}\right) .
	\end{equation*}

	Using properties of Kronecker product, trace and determinants (see, \cite{GN}, \cite{RC}),
	\[\begin{vmatrix} \bm{\Sigma}_{2} \otimes \bm{\Sigma}_{1} \end{vmatrix} = \begin{vmatrix} \bm{\Sigma}_{2}\end{vmatrix}^{p} \begin{vmatrix} \bm{\Sigma}_{1}\end{vmatrix}^{q},\]
	
	\begin{equation*}
		\left(vec(\bm{x})\right)^\top \left(\bm{\Sigma}_{2} \otimes \bm{\Sigma}_{1} \right)^{-1} vec(\bm{x})	= \operatorname{tr}\left(\bm{\Sigma}_{2}^{-1} \bm{x}^\top \bm{\Sigma}_{1}^{-1} \bm{x}\right).
	\end{equation*}
	
	The joint probability density function of $\bm{X}$ and $W$ is
	\[f(\bm{x}, w)=\frac{\exp(-w)} {(2\pi)^\frac{pq}{2}(w)^\frac{pq}{2} \begin{vmatrix} \bm{\Sigma}_{2}\end{vmatrix}^{p/2} \begin{vmatrix} \bm{\Sigma}_{1}\end{vmatrix}^{q/2}} \exp\left(-\frac{1}{2w} \operatorname{tr}\left( \bm{\Sigma}_{2}^{-1} \bm{x}^\top \bm{\Sigma}_{1}^{-1} \bm{x} \right) \right). \]
\end{proof}
 
 \begin{lemma}
 	\label{lemma:2}
 	 If the random matrix $\bm{X}$ and random variable $W$ have a joint density function $f_{\bm{X},W}(\bm{x},w)$ given in the corollary \ref{corly:1}, then the conditional expectation of $\frac{1}{W}$ given $\bm{X}$ is
 	 \begin{equation*}
 	 		E\left(\frac{1}{W} \mid \bm{X}=\bm{x} \right)= \left( \frac{\operatorname{tr}\left(\bm{\Sigma}_{2}^{-1} \bm{x}^\top \bm{\Sigma}_{1}^{-1} \bm{x}\right)}{2}\right)^{-\frac{1}{2}} \frac{K_{\nu-1} \left( \sqrt{ 2 \operatorname{tr}\left(\bm{\Sigma}_{2}^{-1} \bm{x}^\top \bm{\Sigma}_{1}^{-1} \bm{x}\right)}\right)}{K_{\nu} \left( \sqrt{ 2 \operatorname{tr}\left(\bm{\Sigma}_{2}^{-1} \bm{x}^\top \bm{\Sigma}_{1}^{-1} \bm{x}\right)}\right)}.
 	 \end{equation*}
 	 
 \end{lemma}
\begin{proof}
		The conditional distribution of $W$ given $\bm{X}$ is required to find the conditional expectation. The density function of the conditional distribution of $W$ given $\bm{X}$ is obtained as
	\begin{align*}
		f_{W|\bm{X}}(w)&=\frac{f_{\bm{X},W}(\bm{x},w)}{f_{\bm{X}}(\bm{x})}\\
		&= \frac{\exp\left(-w-\frac{1}{2w} \operatorname{tr}\left( \bm{\Sigma}_{2}^{-1} \bm{x}^\top \bm{\Sigma}_{1}^{-1} \bm{x} \right) \right)} {2 (w)^{\frac{pq}{2}} K_{\nu} \left( \sqrt{ 2 \operatorname{tr}\left(\bm{\Sigma}_{2}^{-1} \bm{x}^\top \bm{\Sigma}_{1}^{-1} \bm{x}\right)}\right)}  
		\left( \frac{\operatorname{tr}\left( \bm{\Sigma}_{2}^{-1} \bm{x}^\top \bm{\Sigma}_{1}^{-1} \bm{x}\right)}{2}\right)^{-\frac{\nu}{2}}.
	\end{align*}
		From this density function, the conditional expectation is
	\begin{align*}
		E\left(\frac{1}{W} \mid \bm{X}=\bm{x} \right) &= \int_{0}^{\infty} \frac{1}{w} \, f_{W|\bm{X}} \left(w \right) \, dw \\
		&=   \left( \frac{\operatorname{tr}\left(\bm{\Sigma}_{2}^{-1} \bm{x}^\top \bm{\Sigma}_{1}^{-1} \bm{x}\right)}{2}\right)^{-\frac{1}{2}} \frac{K_{\nu-1} \left( \sqrt{ 2 \operatorname{tr}\left(\bm{\Sigma}_{2}^{-1} \bm{x}^\top \bm{\Sigma}_{1}^{-1} \bm{x}\right)}\right)}{K_{\nu} \left( \sqrt{ 2 \operatorname{tr}\left(\bm{\Sigma}_{2}^{-1} \bm{x}^\top \bm{\Sigma}_{1}^{-1} \bm{x}\right)}\right)}.
	\end{align*}
\end{proof}

		\begin{note}
		If $\bm{X}\sim \mathcal{MSL}_{p,q}(\bm{\Sigma}_{1}, \bm{\Sigma}_{2})$, then the expected value or mean of the random matrix $\bm{X}$ is $\bm{0}$.
	\end{note}

	\begin{theorem}[\textbf{Characteristic function}]
		\label{thm:4}
		If $\bm{X} \sim \mathcal{MSL}_{p,q}(\bm{\Sigma}_{1},\bm{\Sigma}_{2})$, then the characteristic function of $\bm{X}$ is
		\begin{equation}
			\label{eq:6}
			\phi_{\bm{X}}(\bm{T})=\frac{1}{1+\frac{1}{2} \operatorname{tr}\left(\bm{\Sigma}_{2} \bm{T}^\top \bm{\Sigma}_{1} \bm{T}\right)} .
		\end{equation}
	\end{theorem}
	
	\begin{note} If $\bm{\Sigma}_{1}$ and $\bm{\Sigma}_{2}$ are replaced by $a \bm{\Sigma}_{1}$ and   $(1/a) \bm{\Sigma}_{2}$ with $a>0$, respectively, in \eqref{eq:6}, then it does not affect the characteristic function $\phi_{\bm{X}}(\bm{T})$. Therefore, the parameters are defined up to a positive multiplicative constant.
	\end{note}
	
	\section{Maximum likelihood estimation}
		In this section, the MLE of the parameters of multivariate and matrix variate symmetric Laplace distributions are obtained. To obtain these estimators, an iterative algorithm based on the EM algorithm is proposed; since an explicit solution of the score equations is not possible, as the probability density functions of these distributions include the modified Bessel function of the third kind. First, the concept of the EM algorithm in the present context is explained.
	
	\subsection{\textbf{EM algorithm}}
	The EM algorithm is a technique of finding maximum likelihood estimates, in case of missing data (see \cite{DLR}; \cite{MK}). It is an iterative procedure for computing the MLE when the observations can be viewed as incomplete data or the data has unobservable latent variables. In both cases of multivariate and matrix variate symmetric Laplace distribution, $W\sim Exp(1)$ is used as latent variables in the representations $\bm{Y}=\sqrt{W}\bm{Z}$ and $\bm{X}= \sqrt{W} \bm{Z}$. Each iteration of the EM algorithm has two steps, the Expectation step or E-step and the Maximization step or M-step.

	\subsection{\textbf{Maximum likelihood estimation of \texorpdfstring{$\bm{\Sigma}$}{PDFstring} in \texorpdfstring{$\mathcal{SL}_{p}(\bm{\Sigma})$}{PDFstring}}}
	Let $\bm{Y}_1,\bm{Y}_2,\ldots ,\bm{Y}_N$ be random sample from a multivariate symmetric Laplace distribution $\mathcal{SL}_{p}(\bm{\Sigma})$.
	Then, the log-likelihood function (up to an additive constant) is 
	\begin{equation*}
		\ell(\bm{\Sigma})=- \frac{N}{2} \log \begin{vmatrix}\bm{\Sigma} \end{vmatrix} + \frac{\nu}{2} \sum_{i=1}^{N} \log ({\bm{Y}_{i}}^\top \bm{\Sigma}^{-1} \bm{Y}_{i})+ \sum_{i=1}^{N}\log K_{\nu} \left( \sqrt{2 ({\bm{Y}_{i}}^\top \bm{\Sigma}^{-1} \bm{Y}_{i})}  \right).
	\end{equation*}
	The parameter $\bm{\Sigma}$ in the argument of $K_{\nu}$, the modified Bessel function of the third kind, makes maximising this log-likelihood function difficult, as the score equation does not have an explicit solution. Hence, the EM algorithm is used on the joint probability density function of $\bm{Y}$ and $W$ to obtain the MLE of the parameter $\bm{\Sigma}$ using the representation given in the theorem \ref{thm:1}.
	
	The joint probability density function of $\bm{Y}$ and $W$ is
	\[f_{\bm{Y},W}(\bm{y},w)= \frac{\exp(-w)}{(2\pi)^{\frac{p}{2}} \begin{vmatrix} \bm{\Sigma} \end{vmatrix}^{\frac{1}{2}}  w^\frac{p}{2}} \exp\left(-\frac{1}{2w} \bm{y}^\top  \bm{\Sigma}^{-1} \bm{y}\right).\]

   $\left(\bm{Y}_{1},\bm{Y}_{2},\cdots,\bm{Y}_{N},W_{1},W_{2},\cdots, W_{N}\right)$ called the complete data, here, $\bm{Y}_{1},\bm{Y}_{2},\cdots,\bm{Y}_{N}$ are observable data and $W_{1}, W_{2},\ldots W_{N}$ are missing data (latent variables). Thus, using the EM algorithm the MLE of $\bm{\Sigma}$ are obtained as follows:
   
   Using the joint probability density function of $\bm{Y}$ and $W$, the complete data log-likelihood function (up to an additive constant) is 
	
	\begin{equation*}
		\ell_{c}(\bm{\Sigma})= -\frac{N}{2}\log \begin{vmatrix}\bm{\Sigma} \end{vmatrix} -\frac{1}{2}\sum_{i=1}^{N}\frac{1}{W_{i}} \left({\bm{Y}_{i}}^\top \bm{\Sigma}^{-1} \bm{Y}_{i}\right) -\sum_{i=1}^{N}\left(\frac{p}{2} \log W_{i} +W_{i}\right).
	\end{equation*}
	
	Since the last term of this equation does not contain any unknown parameter, it can be ignored for maximization of $\ell_{c}(\bm{\Sigma})$ with respect to $\bm{\Sigma}$.
	Therefore, the function considered for maximization is
	\begin{equation*}
		\ell_{c}(\bm{\Sigma})=-\frac{N}{2}\log \begin{vmatrix}\bm{\Sigma} \end{vmatrix} -\frac{1}{2}\sum_{i=1}^{N}\frac{1}{W_{i}} \left({\bm{Y}_{i}}^\top \bm{\Sigma}^{-1} \bm{Y}_{i}\right).
	\end{equation*}
	
	$W$ is a latent variable, which is not observable, it is replaced with its conditional expectation given $\bm{Y}_1,\bm{Y}_2,\cdots,\bm{Y}_N$ and the current estimate of $\bm{\Sigma}$, (say $\hat{\bm{\Sigma}}$). After taking the conditional expectation, the function to be maximized is 
	\begin{equation}
		\label{eq:10}
		Q(\bm{\Sigma} \mid \bm{Y}_{i},\hat{\bm{\Sigma}} )=-\frac{N}{2}\log \begin{vmatrix}\bm{\Sigma} \end{vmatrix} -\frac{1}{2}\sum_{i=1}^{N} E\left(\frac{1}{W_{i}} | \bm{Y}_{i},\hat{\bm{\Sigma}}\right) \left({\bm{Y}_{i}}^\top \bm{\Sigma}^{-1} \bm{Y}_{i}\right),
	\end{equation}
	where $E(\frac{1}{W_{i}} | \bm{Y}_{i}, \hat{\Sigma})$ is the conditional expectation of $\frac{1}{W_{i}}$ given $\bm{Y}_{i}$ and the current estimate of $\bm{\Sigma}$, that is, $\hat{\bm{\Sigma}}$.

	Thus, from the lemma \ref{lemma:1}, the conditional expectations of $\frac{1}{W_{i}}$ given $\bm{Y}_{i}$ and the current estimate $\hat{\bm{\Sigma}}$, is 
	\begin{equation}
		\label{eq:11}
		v_{i}= E\left(\frac{1}{W_{i}} | \bm{Y}_{i},\hat{\bm{\Sigma}} \right)=  \left( \frac{\bm{Y}_{i}^\top \hat{\bm{\Sigma}}^{-1} \bm{Y}_{i}}{2} \right) ^{-\frac{1}{2}} \frac{K_{\nu-1} \left( \sqrt{2 \left({\bm{Y}_{i}}^\top \hat{\bm{\Sigma}}^{-1} \bm{Y}_{i}\right)}  \right) }{K_{\nu} \left( \sqrt{2 \left({\bm{Y}_{i}}^\top \hat{\bm{\Sigma}}^{-1} \bm{Y}_{i}\right)}  \right) } ,
	\end{equation}
	for $i=1,2,\cdots,N.$
	
	By substituting the conditional expectations in \eqref{eq:10} with $v_{i}$'s as derived in \eqref{eq:11}, the function to be maximized becomes
	\begin{equation}
		\label{eq:12}
		Q(\bm{\Sigma}|\bm{Y}_{i},\hat{\bm{\Sigma}})= -\frac{N}{2}\log \begin{vmatrix}\bm{\Sigma} \end{vmatrix} -\frac{1}{2}\sum_{i=1}^{N} v_{i} {\bm{Y}_{i}}^\top \bm{\Sigma}^{-1} \bm{Y}_{i}.
	\end{equation}
	
	To find the maxima of $\bm{\Sigma}$, differentiating \eqref{eq:12} with respect to $\bm{\Sigma}$ and setting it equal to zero 
	\[ -\frac{N}{2} \bm{\Sigma} ^{-1}+\frac{1}{2} \sum_{i=1}^{N} v_{i} (\bm{\Sigma}^{-1}\bm{Y}_{i}\bm{Y}_{i}^\top \bm{\Sigma}^{-1}) =0.\] 
	The maximum likelihood estimator, the solution of the above score equation, is obtained as
	
	\[ \hat{\bm{\Sigma}} =\frac{1}{N} \sum_{i=1}^{N} v_{i} \bm{Y}_{i}\bm{Y}_{i}^\top . \] 
	
	\textbf{The algorithm for the MLE of $\bm{\Sigma}$ in $\mathcal{SL}_{p}(\bm{\Sigma})$}
	\begin{enumerate}
		\item Set iteration number $k=0$ and select the initial estimate of the parameter $\bm{\Sigma}$, let $\hat{\bm{\Sigma}}_{(0)}$.
		\item Using the current estimates $\hat{\bm{\Sigma}}_{(k-1)}$, for $k=1,2,\cdots$, calculate the conditional expectations
		\[	v_{i}^{(k)}=  \left( \frac{\bm{Y}_{i}^\top \left(\hat{\bm{\Sigma}}_{(k-1)}\right)^{-1} \bm{Y}_{i}}{2} \right)^{-1/2} \frac{K_{\nu-1} \left( \sqrt{2 \left({\bm{Y}_{i}}^\top \left(\hat{\bm{\Sigma}}_{(k-1)}\right)^{-1} \bm{Y}_{i}\right)}  \right) }{K_{\nu} \left( \sqrt{2 \left({\bm{Y}_{i}}^\top \left(\hat{\bm{\Sigma}}_{(k-1)}\right)^{-1} \bm{Y}_{i}\right)}  \right) } , \]
		for $i=1,2,\cdots,N.$
		\item Use the following updated equation to calculate the new estimate
		\[  \hat{\bm{\Sigma}}_{(k)} =\frac{1}{N} \sum_{i=1}^{N} v_{i}^{(k)}\bm{Y}_{i} \bm{Y}_{i}^\top .  \]
		\item Repeat these steps until
		\[\ell \left(\hat{\bm{\Sigma}}_{(k)}\right)-\ell \left(\hat{\bm{\Sigma}}_{(k-1)}\right)  <\epsilon , \  k=1,2,\cdots, \]
		where $\epsilon>0$ is an arbitrary small number and $\ell(\bm{\Sigma})$ is  
		\begin{multline*}
			\ell(\bm{\Sigma})=- \frac{N}{2} \log \begin{vmatrix} \bm{\Sigma} \end{vmatrix} + \frac{\nu}{2} \sum_{i=1}^{N} \log \left(\operatorname{tr}(\bm{\Sigma}^{-1} \bm{Y}_{i} {\bm{Y}_{i}}^\top)\right) + \sum_{i=1}^{N}\log K_{\nu} \left( \sqrt{2 \operatorname{tr}(\bm{\Sigma}^{-1} \bm{Y}_{i}{\bm{Y}_{i}}^\top)}  \right).
		\end{multline*}
	\end{enumerate}
	
	\subsection{\textbf{Maximum likelihood estimation of \texorpdfstring{$\bm{\Sigma}_{1}, \bm{\Sigma}_{2}$}{PDFstring} in \texorpdfstring{$\mathcal{MSL}_{p,q}(\bm{\Sigma}_{1}, \bm{\Sigma}_{2})$}{PDFstring}}}
	
	Let $\bm{X}_1, \bm{X}_2,\ldots , \bm{X}_N$ be random sample from a matrix variate symmetric Laplace distribution $\mathcal{MSL}_{p,q}(\bm{\Sigma}_{1}, \bm{\Sigma}_{2})$. The likelihood function is 
	\begin{equation*}
		L\left( \bm{\Sigma}_{1}, \bm{\Sigma}_{2} \mid \bm{X}_{1},\cdots,\bm{X}_{N} \right) = \prod_{i=1}^{N} f(\bm{X}_{i}), 
	\end{equation*}
	where $f(\bm{X}_{i})$ is as in \eqref{eq:3}.
	Then, the log-likelihood function (up to an additive constant) is 
	\begin{multline*}
		\ell(\bm{\Sigma}_{1}, \bm{\Sigma}_{2}) = -\frac{qN}{2} \log {\begin{vmatrix} \bm{\Sigma}_{1}\end{vmatrix}}-\frac{pN}{2} \log {\begin{vmatrix} \bm{\Sigma}_{2}\end{vmatrix}} +\frac{\nu}{2} \sum_{i=1}^{N} \log \left(\operatorname{tr}\left(\bm{\Sigma}_{2}^{-1}{\bm{X}_{i}}^\top \bm{\Sigma}_{1}^{-1} \bm{X}_{i}\right) \right) \\
		+ \sum_{i=1}^{N}\log K_{\nu} \left( \sqrt{2 \operatorname{tr} \left(\bm{\Sigma}_{2}^{-1} {\bm{X}_{i}}^\top \bm{\Sigma}_{1}^{-1} \bm{X}_{i}\right)}  \right).
	\end{multline*}
	
	Parameters $\bm{\Sigma}_{1}$ and $\bm{\Sigma}_{2}$ in the argument of $K_{\nu}$, the modified Bessel function of the third kind, makes maximising the log-likelihood function difficult, as the score equations do not have explicit solutions. So, the EM algorithm is used on the joint probability density function of $\bm{X}$ and $W$ to obtain the maximum likelihood estimators using the representation given in the theorem \refeq{theorem:3}.
	
	The joint probability density function of $\bm{X}$ and $W$ is
	\[f(\bm{x}, w)=\frac{\exp(-w)} {(2\pi)^\frac{pq}{2}(w)^\frac{pq}{2} \begin{vmatrix} \bm{\Sigma}_{2}\end{vmatrix}^{p/2} \begin{vmatrix} \bm{\Sigma}_{1}\end{vmatrix}^{q/2}} \exp\left(-\frac{1}{2w} \operatorname{tr}\left( \bm{\Sigma}_{2}^{-1} \bm{x}^\top \bm{\Sigma}_{1}^{-1} \bm{x} \right) \right). \]
	
	Suppose $\left(\bm{X}_{1},\bm{X}_{2},\cdots,\bm{X}_{N}, W_{1},W_{2},\cdots,W_{N}\right)$ is the complete data, where $\bm{X}_{1},\cdots,\bm{X}_{N}$ are observable data and $W_{1}, W_{2},\ldots W_{N}$ are missing data (latent variables).
	
	Thus, using the joint probability density function of $\bm{X}$ and $W$, the complete data log-likelihood function (up to an additive constant) is
	
	\begin{multline*}
		\ell_{c}(\bm{\Sigma}_{1}, \bm{\Sigma}_{2})	=-\frac{qN}{2} \log \begin{vmatrix} \bm{\Sigma}_{1}\end{vmatrix}-\frac{pN}{2} \log \begin{vmatrix} \bm{\Sigma}_{2} \end{vmatrix}-\frac{1}{2}\sum_{i=1}^{N}\frac{1}{W_{i}} \operatorname{tr}\left( \bm{\Sigma}_{2}^{-1}{\bm{X}_{i}}^\top \bm{\Sigma}_{1}^{-1} \bm{X}_{i} \right)\\
		-\left(\sum_{i=1}^{N} \left( W_{i} + \frac{pq}{2} \log(W_{i})\right) \right). 
	\end{multline*}

	Since the last term does not contain any unknown parameter, it can be ignored for maximization of $\ell_{c}(\bm{\Sigma}_{1}, \bm{\Sigma}_{2})$ with respect to $\bm{\Sigma}_{1}$ and $\bm{\Sigma}_{2}$.
	Therefore, the function considered for maximization is
	
	\begin{equation*}
		\ell_{c}(\bm{\Sigma}_{1}, \bm{\Sigma}_{2})= -\frac{qN}{2} \log \begin{vmatrix} \bm{\Sigma}_{1}\end{vmatrix}-\frac{pN}{2} \log \begin{vmatrix} \bm{\Sigma}_{2}\end{vmatrix}-\frac{1}{2}\sum_{i=1}^{N}\frac{1}{W_{i}} \operatorname{tr}\left( \bm{\Sigma}_{2}^{-1}{\bm{X}_{i}}^\top \bm{\Sigma}_{1}^{-1} \bm{X}_{i} \right).
	\end{equation*}
	
	$W$ is a latent variable, which is not observable, it is replaced with its conditional expectation given $\bm{X}_{1},\bm{X}_{2}, \cdots, \bm{X}_{N}$ and the current estimates of $\bm{\Sigma}_{1}$ and $\bm{\Sigma}_{2}$, (say $\hat{\bm{\Sigma}}_{1}$ and $\hat{\bm{\Sigma}}_{2}$). Thus, after taking the conditional expectation, the function to be maximized is
	
	\begin{multline}
		\label{eq:16}
		Q(\bm{\Sigma}_{1}, \bm{\Sigma}_{2})= -\frac{qN}{2} \log \begin{vmatrix} \bm{\Sigma}_{1}\end{vmatrix}-\frac{pN}{2} \log \begin{vmatrix} \bm{\Sigma}_{2}\end{vmatrix} \\ -\frac{1}{2}\sum_{i=1}^{N} E\left(\frac{1}{W_{i}} | \bm{X}_{i}, \hat{\bm{\Sigma}}_{1}, \hat{\bm{\Sigma}}_{2}\right) \operatorname{tr}\left( \bm{\Sigma}_{2}^{-1}{\bm{X}_{i}}^\top \bm{\Sigma}_{1}^{-1} \bm{X}_{i} \right),
	\end{multline}
	
	where $E(\frac{1}{W_{i}} | \bm{X}_{i}, \hat{\bm{\Sigma}}_{1}, \hat{\bm{\Sigma}}_{2})$ is the conditional expectation of $\frac{1}{W_{i}}$ given $\bm{X}_{i}$ and the current estimates $\hat{\bm{\Sigma}}_{1}$ and $\hat{\bm{\Sigma}}_{2}$ of $\bm{\Sigma}_{1}$ and $\bm{\Sigma}_{2}$.
	
     Thus, from the lemma \ref{lemma:2}, the conditional expectation of $\frac{1}{W_{i}}$ given $\bm{X}_{i}$ and $\hat{\bm{\Sigma}}_{1}$, $\hat{\bm{\Sigma}}_{2}$, is
	\begin{multline}
		\label{eq:17}
		v_{i}=E\left(\frac{1}{W_{i}} | \bm{X}_{i}, \hat{\bm{\Sigma}}_{1}, \hat{\bm{\Sigma}}_{2}\right) =  \left( \frac{\operatorname{tr}\left( \hat{\bm{\Sigma}}_{2}^{-1} \bm{X}_{i}^\top \hat{\bm{\Sigma}}_{1}^{-1} \bm{X}_{i}\right)}{2}\right)^{-\frac{1}{2}} \\ \frac{K_{\nu-1} \left( \sqrt{ 2 \operatorname{tr}\left(\hat{\bm{\Sigma}}_{2}^{-1} \bm{X}_{i}^\top \hat{\bm{\Sigma}}_{1}^{-1} \bm{X}_{i}\right)}\right)}{K_{\nu} \left( \sqrt{ 2 \operatorname{tr}\left(\hat{\bm{\Sigma}}_{2}^{-1} \bm{X}_{i}^\top \hat{\bm{\Sigma}}_{1}^{-1} \bm{X}_{i}\right)}\right)},
	\end{multline}
	for $i=1,2,\cdots,N.$
	
	By substituting the conditional expectations in \eqref{eq:16} with $v_{i}$'s as derived in \eqref{eq:17}, the function to be maximized becomes
	\begin{multline}
		\label{eq:18}
		Q(\bm{\Sigma}_{1}, \bm{\Sigma}_{2}|\bm{X}_{i}, \hat{\bm{\Sigma}}_{1}, \hat{\bm{\Sigma}}_{2})= -\frac{qN}{2} \log \begin{vmatrix} \bm{\Sigma}_{1}\end{vmatrix}-\frac{pN}{2} \log \begin{vmatrix} \bm{\Sigma}_{2}\end{vmatrix}\\ -\frac{1}{2}\sum_{i=1}^{N} v_{i} \operatorname{tr}\left( \bm{\Sigma}_{2}^{-1}{\bm{X}_{i}}^\top \bm{\Sigma}_{1}^{-1} \bm{X}_{i} \right).
	\end{multline}
	To find the maxima of $\bm{\Sigma_{1}}$ and $\bm{\Sigma_{2}}$, differentiating \eqref{eq:18} with respect to $\bm{\Sigma}_{1}$ and $\bm{\Sigma}_{2}$ (for the matrix derivatives, see \cite{Dw}, \cite{Gr}) and setting them equal to zero, the score equations obtained are
	\begin{multline*}
		\diffp{Q}{{\bm{\Sigma}_{1}}}=-qN \bm{\Sigma}_{1}^{-1}+\frac{qN}{2} \operatorname{diag} \left( \bm{\Sigma}_{1}^{-1} \right) +\bm{\Sigma}_{1}^{-1}\left(\sum_{i=1}^{N} v_{i} \bm{X}_{i} \bm{\Sigma}_{2}^{-1} {\bm{X}_{i}}^\top \right) \bm{\Sigma}_{1}^{-1} \\
		-\frac{1}{2}\operatorname{diag}\left( \bm{\Sigma}_{1}^{-1}\left(\sum_{i=1}^{N} v_{i} \bm{X}_{i}\bm{\Sigma}_{2}^{-1} {\bm{X}_{i}}^\top \right) \bm{\Sigma}_{1}^{-1}\right)=0,
	\end{multline*}
	\begin{multline*}
		\diffp{Q} {{\bm{\Sigma}_{2}}} =	-pN \bm{\Sigma}_{2}^{-1}+\frac{pN}{2} \operatorname{diag}(\bm{\Sigma}_{2}^{-1})+\bm{\Sigma}_{2}^{-1}\left(\sum_{i=1}^{N} v_{i} \bm{X}_{i}^\top \bm{\Sigma}_{1}^{-1} {\bm{X}_{i}}\right) \bm{\Sigma}_{2}^{-1} \\
		-\frac{1}{2}\operatorname{diag}\left( \bm{\Sigma}_{2}^{-1}\left(\sum_{i=1}^{N} v_{i} \bm{X}_{i}^\top \bm{\Sigma}_{1}^{-1} {\bm{X}_{i}} \right) \bm{\Sigma}_{2}^{-1}\right)=0.
	\end{multline*}

	The maximum likelihood estimators, solutions of the above score equations, are obtained as
	
	\begin{equation}
		\label{eq:19}
		\hat{\bm{\Sigma}}_{1}=\frac{1}{qN}\sum_{i=1}^{N} v_{i} \bm{X}_{i} \bm{\Sigma}_{2}^{-1} {\bm{X}_{i}}^\top,
	\end{equation}  
	\begin{equation}
		\label{eq:20}
		\hat{\bm{\Sigma}}_{2}=\frac{1}{pN}\sum_{i=1}^{N} v_{i} \bm{X}_{i}^\top \bm{\Sigma}_{1}^{-1} {\bm{X}_{i}}.
	\end{equation}
	
	\textbf{The algorithm for the MLE of $\bm{\Sigma}_{1}$ and $\bm{\Sigma}_{2}$ in $\mathcal{MSL}_{p,q}(\bm{\Sigma}_{1}, \bm{\Sigma}_{2})$}
	\begin{enumerate}
		\item Set iteration number $k=0$ and select the initial estimate of the parameters $\bm{\Sigma}_{1}$ and $\bm{\Sigma}_{2}$, let $\hat{\bm{\Sigma}}_{1}^{(0)}$ and $\hat{\bm{\Sigma}}_{2}^{(0)}$, respectively.
		\item Using the current estimates $\hat{\bm{\Sigma}}_{1}^{(k-1)}$ and $\hat{\bm{\Sigma}}_{2}^{(k-1)}$, for $k=1,2,\cdots$, calculate the conditional expectations
		
		\begin{multline*}
			v_{i}^{(k)}=  \left( \frac{\operatorname{tr}\left( \left(\hat{\bm{\Sigma}}_{2}^{(k-1)}\right)^{-1} \bm{X}_{i}^\top \left(\hat{\bm{\Sigma}}_{1}^{(k-1)}\right)^{-1} \bm{X}_{i}\right)}{2}\right)^{-\frac{1}{2}} \\ 
			\frac{K_{\nu-1} \left( \sqrt{ 2 \operatorname{tr}\left(\left(\hat{\bm{\Sigma}}_{2}^{(k-1)}\right)^{-1} \bm{X}_{i}^\top \left(\hat{\bm{\Sigma}}_{1}^{(k-1)}\right)^{-1} \bm{X}_{i}\right)}\right)}{K_{\nu} \left( \sqrt{ 2 \operatorname{tr}\left(\left(\hat{\bm{\Sigma}}_{2}^{(k-1)}\right)^{-1} \bm{X}_{i}^\top \left(\hat{\bm{\Sigma}}_{1}^{(k-1)}\right)^{-1} \bm{X}_{i}\right)}\right)},
		\end{multline*}
		
		for $i=1,2,\cdots,N.$
		\item Use the following updated equations to calculate the new estimate
		\[ \hat{\bm{\Sigma}}_{1}^{(k)} =\frac{1}{qN}\sum_{i=1}^{N} v_{i}^{(k)} \bm{X}_{i} \left(\hat{\bm{\Sigma}}_{2}^{(k-1)}\right)^{-1} {\bm{X}_{i}}^\top, \]
		\[\hat{\bm{\Sigma}}_{2}^{(k)}=\frac{1}{pN}\sum_{i=1}^{N} v_{i}^{(k)} \bm{X}_{i}^\top \left(\hat{\bm{\Sigma}}_{1}^{(k)}\right)^{-1} {\bm{X}_{i}}. \]
		\item Repeat these steps until
		\[\ell \left(\hat{\bm{\Sigma}}_{1}^{(k)}, \hat{\bm{\Sigma}}_{2}^{(k)}\right)-\ell \left(\hat{\bm{\Sigma}}_{1}^{(k-1)},\hat{\bm{\Sigma}}_{2}^{(k-1)}\right)  <\epsilon , \  k=1,2,\cdots, \]
		where $\epsilon>0$ is an arbitrary small number and $\ell(\bm{\Sigma}_{1}, \bm{\Sigma}_{2})$ is
		\begin{multline*}
			\ell(\bm{\Sigma}_{1}, \bm{\Sigma}_{2}) = -\frac{qN}{2} \log {\begin{vmatrix} \bm{\Sigma}_{1}\end{vmatrix}}-\frac{pN}{2} \log {\begin{vmatrix} \bm{\Sigma}_{2}\end{vmatrix}} \\
			+\frac{\nu}{2} \sum_{i=1}^{N} \log \left(\operatorname{tr}\left({\bm{\Sigma}_{2}}^{-1}{\bm{X}_{i}}^\top {\bm{\Sigma}_{1}}^{-1} \bm{X}_{i}\right) \right) \\
			+ \sum_{i=1}^{N}\log K_{\nu} \left( \sqrt{2 \operatorname{tr} ({\bm{\Sigma}_{2}}^{-1} {\bm{X}_{i}}^\top {\bm{\Sigma}_{1}}^{-1} \bm{X}_{i})}  \right).
		\end{multline*}
		
	\end{enumerate}
	In the next section, the existence of the proposed MLE are discussed.

	\section{Existence of estimators} 
	It is claimed that maximum likelihood estimators exist for the parameters $\bm{\Sigma}_{1}, \bm{\Sigma}_{2}$ of the matrix variate symmetric Laplace distribution if the sample size 
	\begin{equation*}
		N\ge \max \left(\frac{p}{q}, \frac{q}{p}\right).
	\end{equation*}
	If $q=1$, it reduces to the multivariate symmetric Laplace distribution with scale parameter $\bm{\Sigma}_{1}$. Hence, first, this claim is validated for $q=1$.
	
	\begin{theorem}
		\label{theorem:5}
		Let $\bm{Y}_1,\bm{Y}_2,\cdots \bm{Y}_N \overset{\text{i.i.d}}{\sim} \mathcal{SL}_{p}(\bm{\Sigma})$, then  maximum likelihood estimator exists almost surely for the parameters $\bm{\Sigma}$ of multivariate symmetric Laplace distribution if and only if the sample size 
		\[N\ge p.\]
	\end{theorem}
	\begin{proof}
		Consider the matrix
		\[\bm{Y}= \begin{pmatrix}
			\sqrt{v_{1}} {\bm{Y}_{1}},
			\sqrt{v_{2}} {\bm{Y}_{2}},
			\cdots 
			\sqrt{v_{N}} {\bm{Y}_{N}}
		\end{pmatrix} \in \mathbb{R}^{p\times N},\] 
		where $v_{1}, v_{2},\cdots, v_{N}$ are the conditional expectations which is calculated in step 2 in section 3.2 and all the $v_{i}$'s are positive real numbers. 
		
         Consider,
		\[\bm{A}=\sum_{i=1}^{N} v_{i}\bm{Y}_{i} {\bm{Y}_{i}}^\top =  \bm{Y} \bm{Y}^\top,\]
		and $rank( \bm{Y} \bm{Y}^\top)= rank(\bm{Y})$. 
		
		It will be sufficient to show that $\bm{Y}$ has rank $p$ with probability $1$ if and only if $N\ge p$, which is equal to the fact that $rank(\bm{Y})<p$, if $N<p$. Thus, it will be sufficient to show that $\bm{Y}$ has rank $p$ with probability $1$, when $N=p$.
		
		 Let $\{x_{1},x_{2},\cdots,x_{p-1}\}$ be a set of vectors in $\mathbb{R}^p$ and let $S\{x_{i} | i=1,2,\cdots,p-1\}$ be the subspace spanned by $x_{1},x_{2},\cdots,x_{p-1}$. Now, $\Pr[\bm{Y}_{i}\in S\{x_{i} | i=1,2,\cdots,p-1\}]=0$, as $\Sigma_{p\times p}$ is positive definite matrix. Let $F$ be the joint density function of $\bm{Y}_{2},\bm{Y}_{3},\cdots,\bm{Y}_{p}$.
		Now,
    \begin{align*}
			P\left(rank(\bm{Y})<p\right)&=P\left[\sqrt{v_{1}} {\bm{Y}_{1}}, \sqrt{v_{2}} {\bm{Y}_{2}}, \cdots 	\sqrt{v_{N}} {\bm{Y}_{p}} \hspace{.7mm}\text{are linearly dependent}\right]\\
	     &=P[\bm{Y}_1,\bm{Y}_2,\cdots,\bm{Y}_{p}\hspace{.7mm}\text{are linearly dependent}]\\
			  &\le \sum_{i=1}^{p}P[\bm{Y}_{i}\in S\{\bm{Y}_1,\bm{Y}_2,\cdots,\bm{Y}_{i-1},\bm{Y}_{i+1},\cdots,\bm{Y}_{p}\}]\\
			  &=p P[\bm{Y}_{1}\in S\{\bm{Y}_2,\bm{Y}_3,\cdots,\bm{Y}_{p}\}]\\
			  &= p \int_{\mathbb{R}^{p(p-1)}}P[\bm{Y}_{1}\in S\{\bm{Y}_2,\bm{Y}_3,\cdots,\bm{Y}_{p}\}| \bm{Y}_2,\bm{Y}_3,\cdots,\bm{Y}_{p}]\, dF(\bm{Y}_2,\cdots,\bm{Y}_{p})\\
			  &=p   \int_{\mathbb{R}^{p(p-1)}}P[\bm{Y}_{1}\in S\{\bm{Y}_2,\bm{Y}_3,\cdots,\bm{Y}_{p}\}| \bm{Y}_2,\bm{Y}_3,\cdots,\bm{Y}_{p}] \, dF\\
			  &=p   \int_{\mathbb{R}^{p(p-1)}} 0\, dF\\
			  &=0.
		\end{align*}
    	Therefore, $\hat{\bm{\Sigma}}=\frac{1}{N}\sum_{i=1}^{N} v_{i}\bm{Y}_{i} {\bm{Y}_{i}}^\top$ is positive definite with probability $1$ if and only if $N\ge p$, hence the MLE exists almost surely.
	\end{proof}

	\begin{theorem}
		\label{theorem:6}
		Let $\bm{X}_1,\bm{X}_2,\ldots ,\bm{X}_N \overset{\text{i.i.d}}{\sim} \mathcal{MSL}_{p,q}(\bm{\Sigma}_{1},\bm{\Sigma}_{2})$, then  maximum likelihood estimators exists almost surely for the parameters $\bm{\Sigma}_{1}, \bm{\Sigma}_{2}$ of matrix variate symmetric Laplace distribution if and only if the sample size 
		\begin{equation*}
			N\ge \max \left(\frac{p}{q}, \frac{q}{p}\right).
		\end{equation*}
	\end{theorem}
	
	\begin{proof}
		Consider the maximum likelihood estimators of $\bm{\Sigma}_{1}$ and $\bm{\Sigma}_{2}$, from the EM algorithm, 
		\begin{equation*}
			\hat{\bm{\Sigma}}_{1}=\frac{1}{qN}\sum_{i=1}^{N} v_{i} \bm{X}_{i}\tilde{\bm{\Sigma}}_{2}^{-1} {\bm{X}_{i}}^\top,
		\end{equation*}  
		\begin{equation*}
			\hat{\bm{\Sigma}}_{2}=\frac{1}{pN}\sum_{i=1}^{N} v_{i} \bm{X}_{i}^\top \tilde{\bm{\Sigma}}_{1}^{-1} {\bm{X}_{i}},
		\end{equation*}
		where $\tilde{\bm{\Sigma}}_{1}, \tilde{\bm{\Sigma}}_{2}$ are the initial estimates of $\bm{\Sigma}_{1}$ and $\bm{\Sigma}_{2}$, respectively, and $v_{i}$'s are the conditional expectations, which are calculated in the second step of this algorithm, and all the $v_{i}$'s are positive real numbers.
		Now, rewrite the above equations in matrix notation, with $\bm{I}_{N}$ is $N\times N$ identity matrix,
		\begin{multline*}
			\hat{\bm{\Sigma}}_{1}=\frac{1}{qN}\begin{pmatrix}
				\sqrt{v_1}\bm{X}_{1} &\sqrt{v_2} \bm{X}_{2} & \cdots & \sqrt{v_N}\bm{X}_{N}
			\end{pmatrix}
			\left\{ \bm{I}_{N}\otimes \tilde{\bm{\Sigma}}_{2}^{-1}\right\} \\ \begin{pmatrix}
				\sqrt{v_1}\bm{X}_{1} &\sqrt{v_2} \bm{X}_{2} & \cdots & \sqrt{v_N}\bm{X}_{N}
			\end{pmatrix}^\top,
		\end{multline*}
		\begin{multline*}
			\hat{\bm{\Sigma}}_{2}=\frac{1}{pN}\begin{pmatrix}
				\sqrt{v_1}\bm{X}_{1}^\top &\sqrt{v_2} \bm{X}_{2}^\top & \cdots & \sqrt{v_N}\bm{X}_{N}^\top
			\end{pmatrix}
			\left\{ \bm{I}_{N}\otimes \tilde{\bm{\Sigma}}_{1}^{-1}\right\} \\
			\begin{pmatrix}
				\sqrt{v_1}\bm{X}_{1}^\top &\sqrt{v_2} \bm{X}_{2}^\top & \cdots & \sqrt{v_N} \bm{X}_{N}^\top
			\end{pmatrix}^\top.
		\end{multline*}
		
		Thus, the matrices $\hat{\bm{\Sigma}}_{1}$ and $\hat{\bm{\Sigma}}_{2}$ are quadratic forms in \[\begin{pmatrix}
			\sqrt{v_1}\bm{X}_{1} &\sqrt{v_2} \bm{X}_{2} & \cdots & \sqrt{v_N}\bm{X}_{N}
		\end{pmatrix}\] and
		\[\begin{pmatrix}
			\sqrt{v_1}\bm{X}_{1}^\top &\sqrt{v_2} \bm{X}_{2}^\top & \cdots & \sqrt{v_N}\bm{X}_{N}^\top
		\end{pmatrix},\] respectively, and the rank of these matrices satisfies the following conditions
		\begin{center}
			$rank(\hat{\bm{\Sigma}}_{1}) =rank \left\{ \bm{I}_{N}\otimes \tilde{\bm{\Sigma}}_{2}^{-1}\right\} =N q$ if and only if $\tilde{\bm{\Sigma}}_{2}$ is positive definite with probability $1$;
		\end{center}
		\begin{center}
			$rank(\hat{\bm{\Sigma}}_{2}) =rank \left\{ \bm{I}_{N}\otimes \tilde{\bm{\Sigma}}_{1}^{-1}\right\} =N p$ if and only if $\tilde{\bm{\Sigma}}_{1}$ is positive definite with probability $1$.
		\end{center}
		Hence, maximum likelihood estimators $\hat{\bm{\Sigma}}_{1}$ and $\hat{\bm{\Sigma}}_{2}$ are positive definite with probability $1$ if and only if $Nq\ge p$ and $Np\ge q$, or $N\ge\frac{p}{q}$ and $N\ge \frac{q}{p}$, which implies that $N\ge \max \left(\frac{p}{q}, \frac{q}{p}\right)$.
		
	\end{proof}
	
  \begin{note}
	Theorem \ref{theorem:5} is a special case of the theorem \refeq{theorem:6} with $q=1$.
 \end{note}

\section{The performance of the proposed MLE} 
	
	\subsection{Comparision of estimators of \texorpdfstring{$\bm{\Sigma}$}{PDFstring} in \texorpdfstring{$\mathcal{SL}_{p}(\bm{\Sigma})$}{PDFstring}}
	In this subsection, we compare the EM estimator with another estimator of $\bm{\Sigma}$, proposed by \cite{EKL}, \cite{KS}
	\begin{equation}
		\label{eq:25}
		\bm{\Sigma}^* = \frac{1}{N-1} \sum_{i=1}^{N} (\bm{Y}_{i}-\overline{\bm{Y}}) (\bm{Y}_{i}-\overline{\bm{Y}})^{\top},
	\end{equation}
 where $N$ is sample size and $\overline{\bm{Y}}$ is the empirical mean of the data.

The performance of estimators is compared with respect to the bias and the mean Euclidean distance
\begin{enumerate}
	\item Empirical bias :- $\| (\hat{\bm{\Sigma}}) _{m}- \bm{\Sigma}\|_{2}$.
	\item The mean Euclidean distance :- $ \| \hat{\bm{\Sigma}}- \bm{\Sigma}\|_{2,m}.$
\end{enumerate}
	$(\hat{\bm{\Sigma}})_{m}$ denotes the empirical mean of estimate of $ \bm{\Sigma}$ over all simulations, and $\|.\|_{2,m}$ denotes the empirical mean of norms over all simulations. For simulation, three types of structures are considered for $\bm{\Sigma}$ named as Cases $1$-$6$: diagonal, block diagonal and full matrices, for $p=6$ and $p=10$. The sample size ($N$) is taken as $10, 20, 30, 50, 70, 100, 150$ and $200$. The number of simulation runs, $s$ is $200$ for all sample sizes, and $\epsilon=10^{-11}$ 
  
     \textbf{Case $1.$} $\bm{\Sigma} =  \scalebox{0.75} {$ \setlength{\arraycolsep}{3pt}\begin{bmatrix}
			 5&0&0&0&0&0\\0&4&0&0&0&0\\ 0&0&3.5&0&0&0\\0&0&0&3&0&0\\0&0&0&0&2&0\\0&0&0&0&0&1
		\end{bmatrix}$},$  \textbf{Case $2.$}  $\bm{\Sigma}=  \scalebox{0.75} {$\setlength{\arraycolsep}{3pt} \begin{bmatrix} 3&1.5&1&0&0&0\\1.5&2&0.5&0&0&0\\1&0.5&1&0&0&0\\0&0&0&4&1&2\\0&0&0&1&5&3\\0&0&0&2&3&6 \end{bmatrix}$},$ 
	  
	\textbf{Case $3.$} $\bm{\Sigma} = \scalebox{0.75} {$\setlength{\arraycolsep}{3pt} \begin{bmatrix}20&3&2&1&4&5\\3&25&6&2&3&1\\2&6&30&7&5&4\\1&2&7&35&6&3\\4&3&5&6&40&8\\5&1&4&3&8&45 \end{bmatrix}$},$ 
	
    \textbf{Case $4.$}$\bm{\Sigma} =  \scalebox{0.85} {$\setlength{\arraycolsep}{3pt} \begin{bmatrix} 6&0&0&0&0&0&0&0&0&0\\0&5.5&0&0&0&0&0&0&0&0\\ 0&0&5&0&0&0&0&0&0&0\\0&0&0&4&0&0&0&0&0&0\\0&0&0&0&3.5&0&0&0&0&0\\0&0&0&0&0&3&0&0&0&0\\0&0&0&0&0&0&2.5&0&0&0\\0&0&0&0&0&0&0&2&0&0\\ 0&0&0&0&0&0&0&0&1.5&0\\0&0&0&0&0&0&0&0&0&1 \end{bmatrix}$},$ 
	
 	\textbf{Case $5.$} $\bm{\Sigma} =  \scalebox{0.85} {$\setlength{\arraycolsep}{3.5pt}\begin{bmatrix}	5&3&2.5&2&1.5&0&0&0&0&0\\3&4&2&1.5&1&0&0&0&0&0\\2.5&2&3&1&0.5&0&0&0&0&0\\2&1.5&1&2&0.2&0&0&0&0&0\\1.5&1&0.5&0.2&1&0&0&0&0&0\\0&0&0&0&0&6&2&1&0.5&1.5\\0&0&0&0&0&2&5&1.2&0.8&1\\ 0&0&0&0&0&1&1.2&4&1&0.6\\0&0&0&0&0&0.5&0.8&1&3.5&0.9\\0&0&0&0&0&1.5&1&0.6&0.9&4\end{bmatrix}$},$
	
  	\textbf{Case $6.$} $\bm{\Sigma} =  \scalebox{0.95} {$\setlength{\arraycolsep}{5.5pt} \begin{bmatrix} 10 & 2 & 1& 0.5 & 1 & 1.5 & 0.8 & 1.2 & 0.9 & 0.7 \\ 2 & 9 & 1.5 & 0.7 & 1.1 & 1.3 & 0.6 & 1 & 0.8 & 0.5\\ 1  & 1.5 & 8 & 1.2 & 0.9 & 0.7 & 1 & 1.1 & 0.6 & 0.4\\ 0.5 & 0.7 & 1.2 & 7 & 1.3 & 0.9 & 1.1 & 0.5 & 0.4 & 0.6\\ 1 & 1.1 & 0.9 & 1.3 & 9 & 1.2 & 1.4 & 0.8 & 1 & 0.7\\ 1.5 & 1.3 & 0.7 & 0.9 & 1.2 & 10 & 0.9 & 1.1 & 0.6 & 0.8\\ 0.8 & 0.6 & 1 & 1.1 & 1.4 & 0.9 & 8 & 1.3 & 1.2 & 0.5\\ 1.2 & 1 & 1.1 & 0.5 & 0.8 & 1.1 & 1.3 & 9 & 1 & 0.6\\ 0.9 & 0.8 & 0.6 & 0.4 & 1 & 0.6 & 1.2 & 1 & 8 & 0.7\\ 0.7 & 0.5 & 0.4 & 0.6 & 0.7 & 0.8 & 0.5 & 0.6 & 0.7 & 7 \end{bmatrix}$}.$

	In all the cases, the initial value $\hat{\bm{\Sigma}}_{(0)}$ for the EM estimator is taken as 
		\[ \hat{\bm{\Sigma}}_{(0)} =\frac{1}{N} \sum_{i=1}^{N} \bm{Y}_{i} \bm{Y}_{i}^\top . \]
		                         
	\begin{figure}[htbp]
		\centering
		\includegraphics[width=1\linewidth]{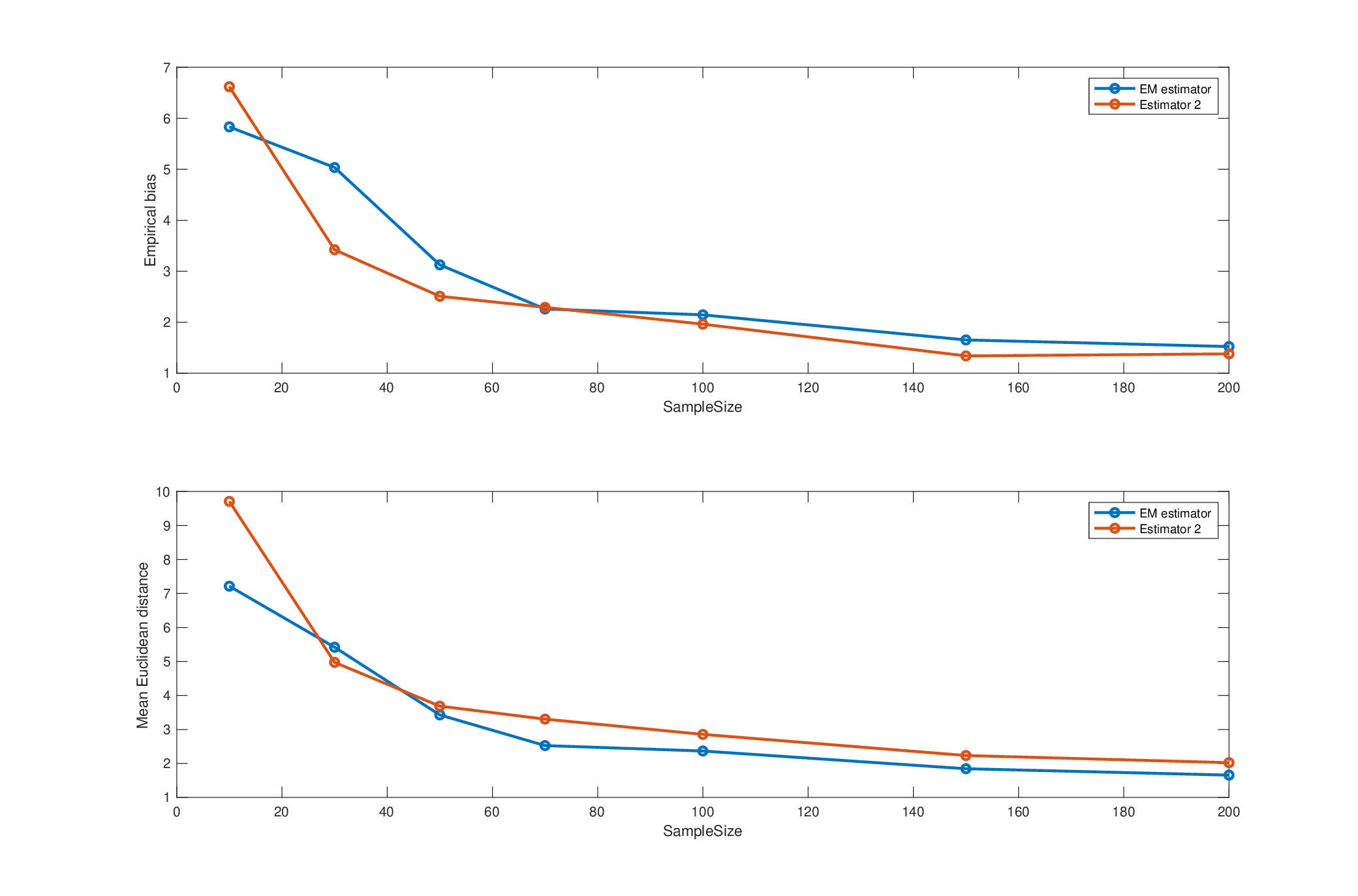}
		\caption{\textbf{Bias and mean Euclidean distance analysis of estimators $\hat{\bm{\Sigma}}$ and $\bm{\Sigma}^*$ over s simulation runs with respect to the sample size, for Case $1$.}}
		\label{fig:case1}
		\centering
		\includegraphics[width=1\linewidth]{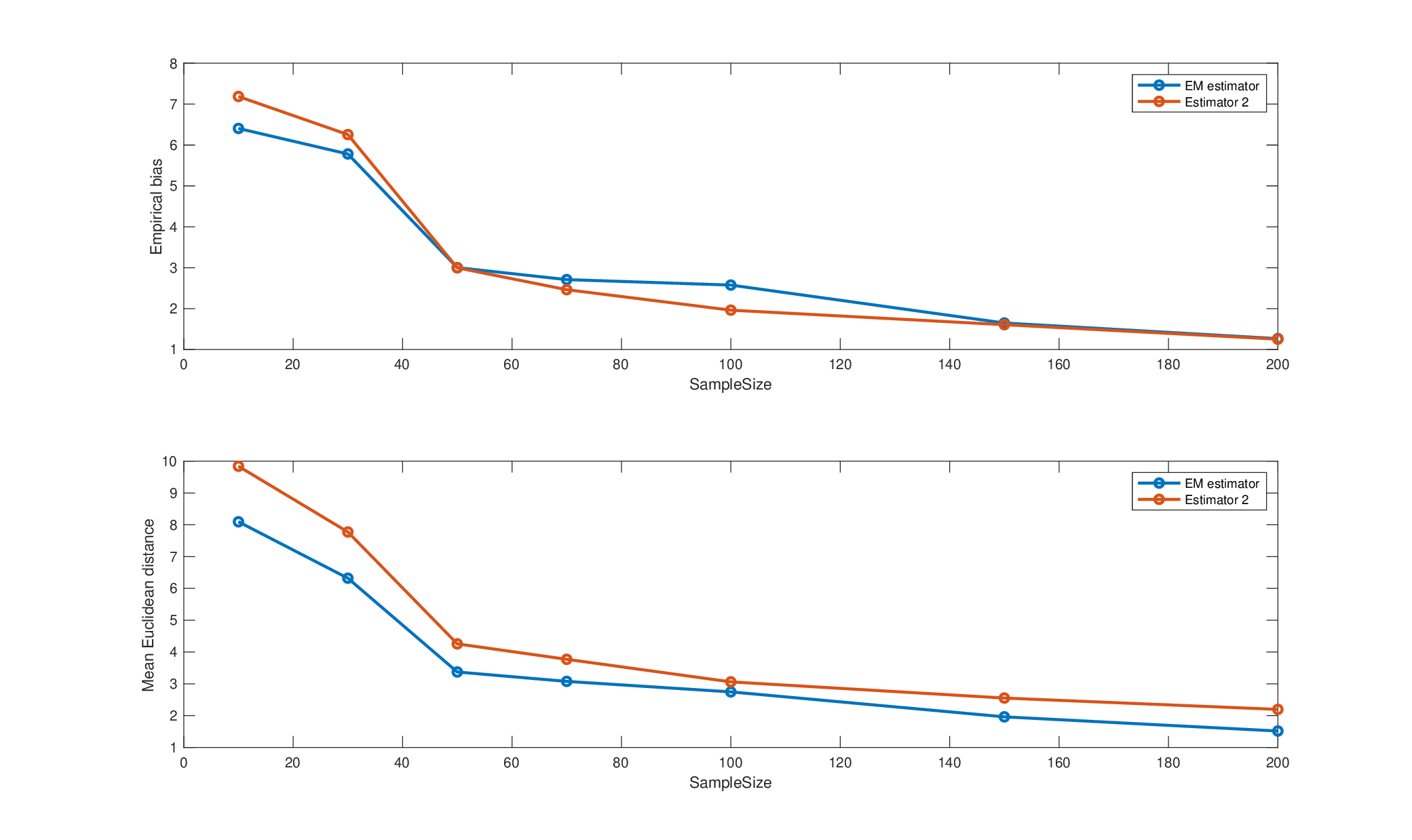}
		\caption{\textbf{Bias and mean Euclidean distance analysis of estimators $\hat{\bm{\Sigma}}$ and $\bm{\Sigma}^*$ over s simulation runs with respect to the sample size, for Case $2$.}}
		\label{fig:case2}
  \end{figure}

\begin{figure}[htbp]
		\centering
		\includegraphics[width=1\linewidth]{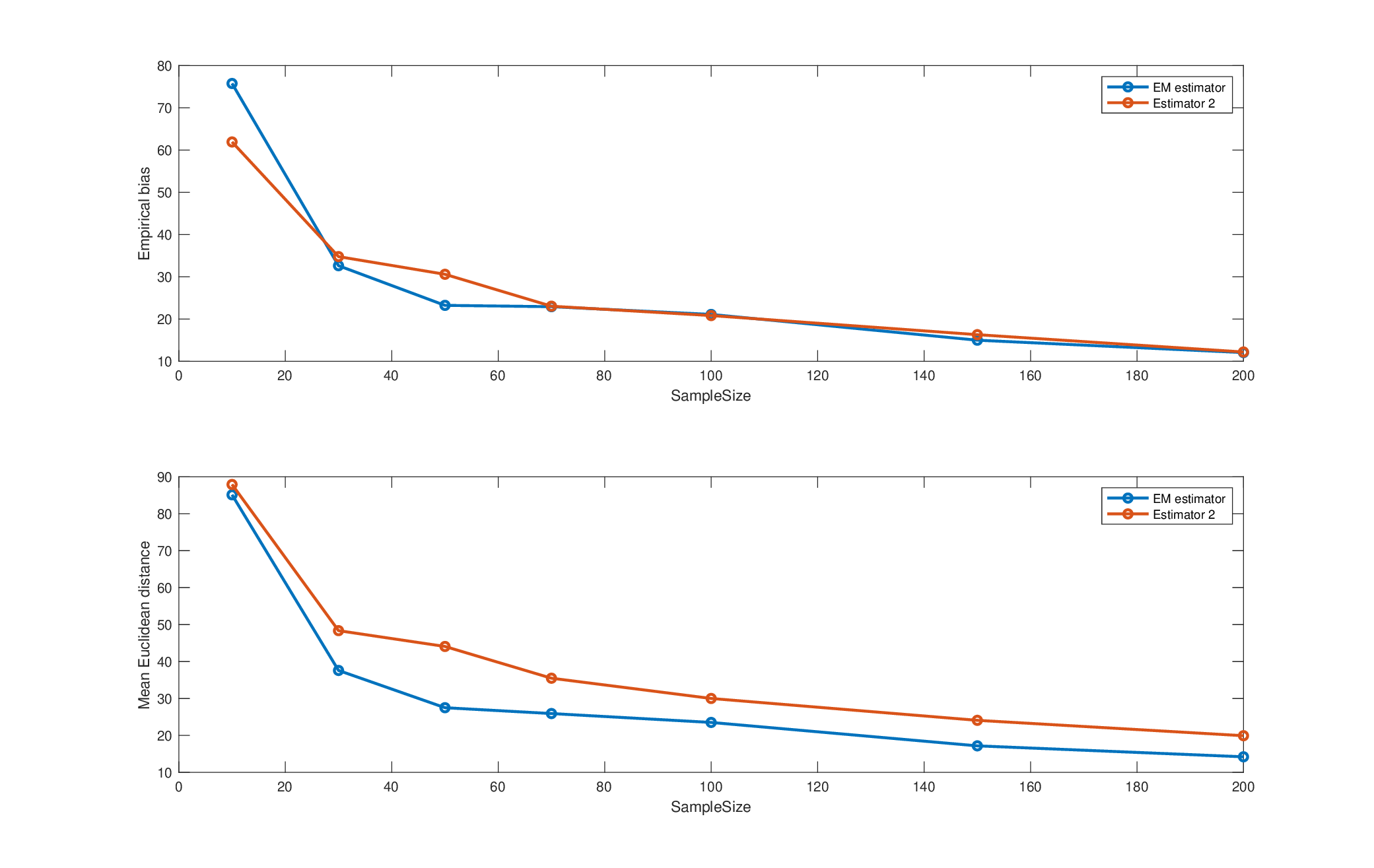}
		\caption{\textbf{Bias and mean Euclidean distance analysis of estimators $\hat{\bm{\Sigma}}$ and $\bm{\Sigma}^*$ over s simulation runs with respect to the sample size, for Case $3$.}}
		\label{fig:case3}
		\centering
		\includegraphics[width=1\linewidth]{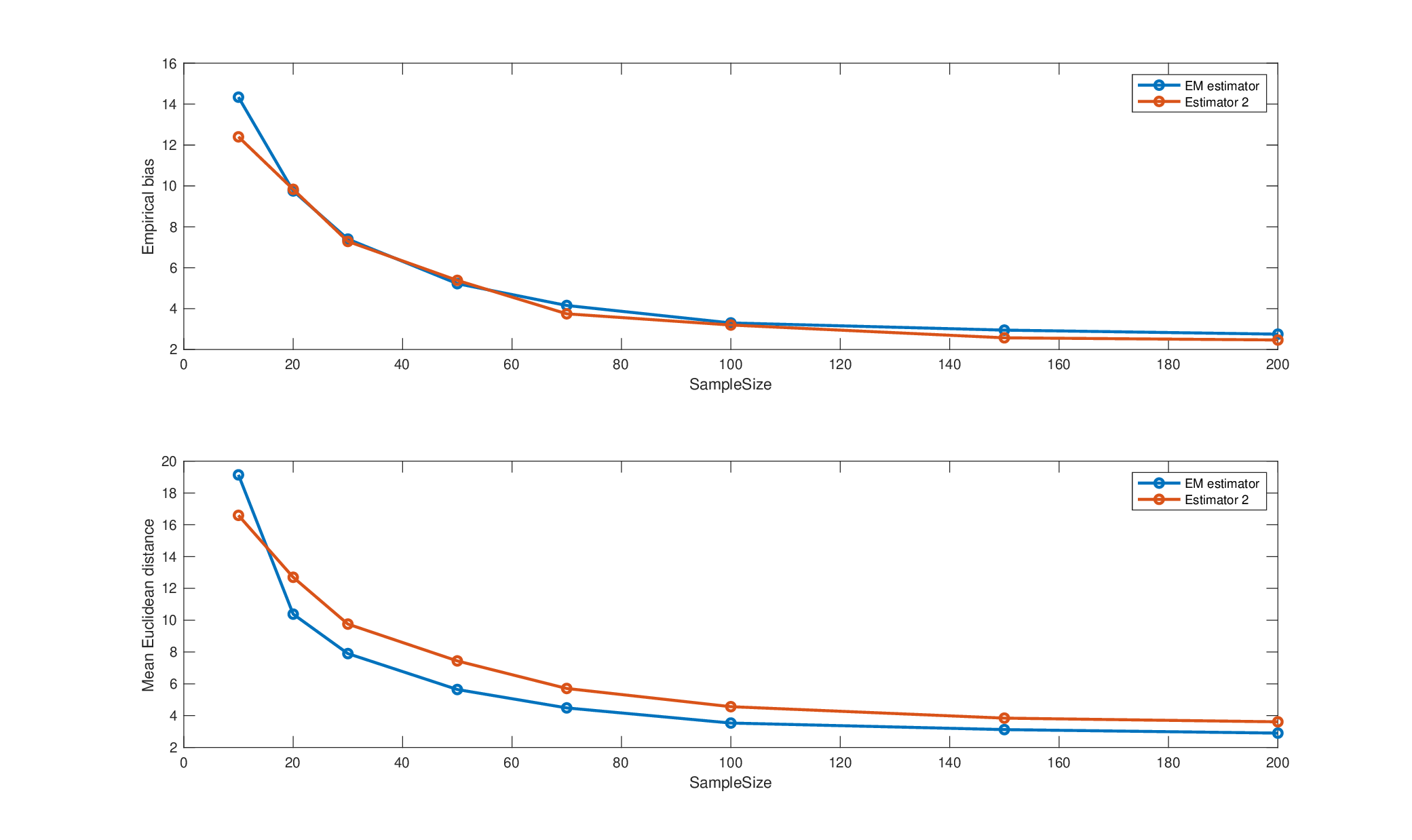}
		\caption{\textbf{Bias and mean Euclidean distance analysis of estimators $\hat{\bm{\Sigma}}$ and $\bm{\Sigma}^*$ over s simulation runs with respect to the sample size, for Case $4$.}}
		\label{fig:case4}
\end{figure}

\begin{figure}[htbp]
		\centering
		\includegraphics[width=1\linewidth]{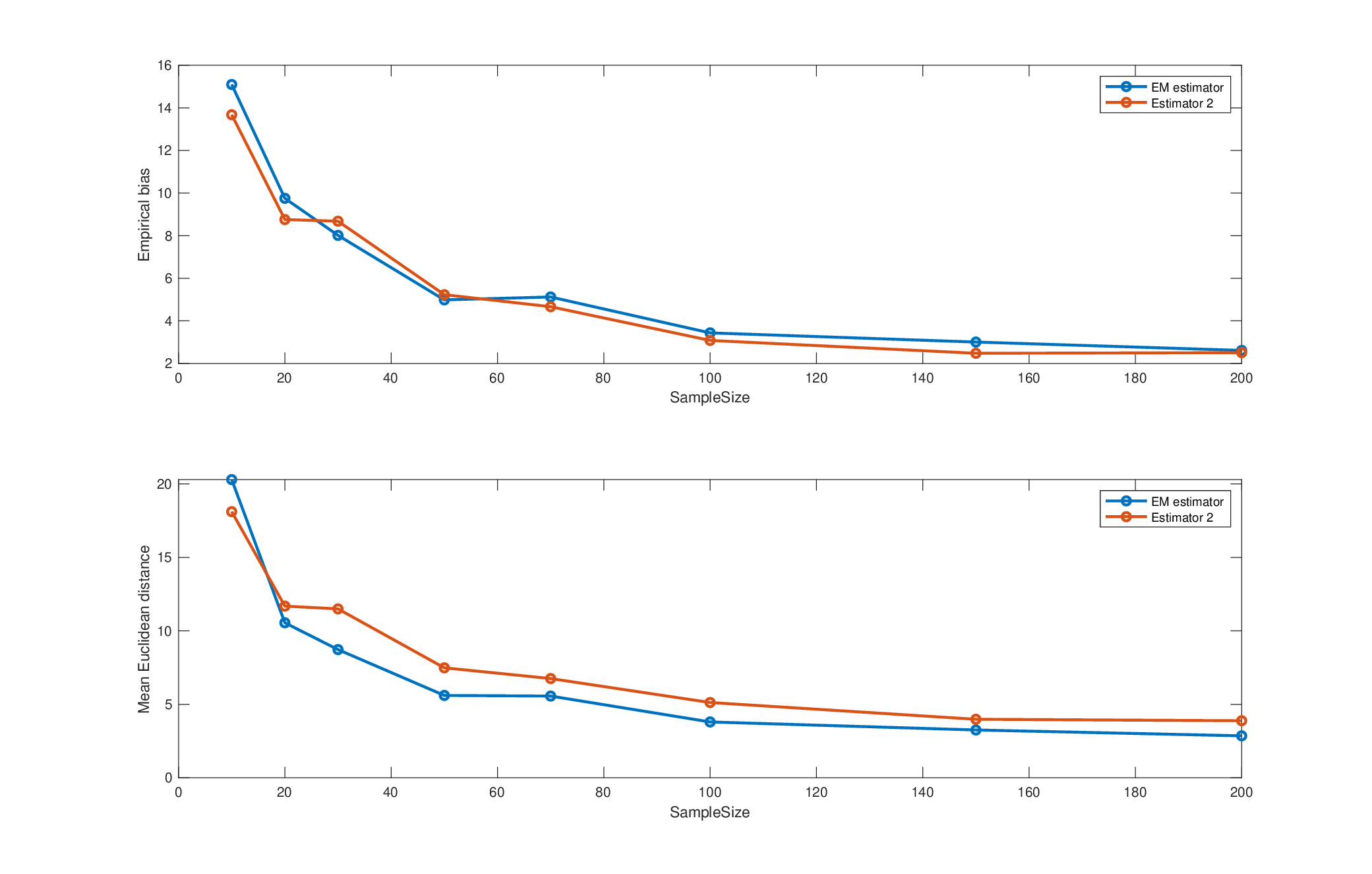}
		\caption{\textbf{Bias and mean Euclidean distance analysis of estimators $\hat{\bm{\Sigma}}$ and $\bm{\Sigma}^*$ over s simulation runs with respect to the sample size, for Case $5$.}}
		\label{fig:case5}

		\centering
		\includegraphics[width=1\linewidth]{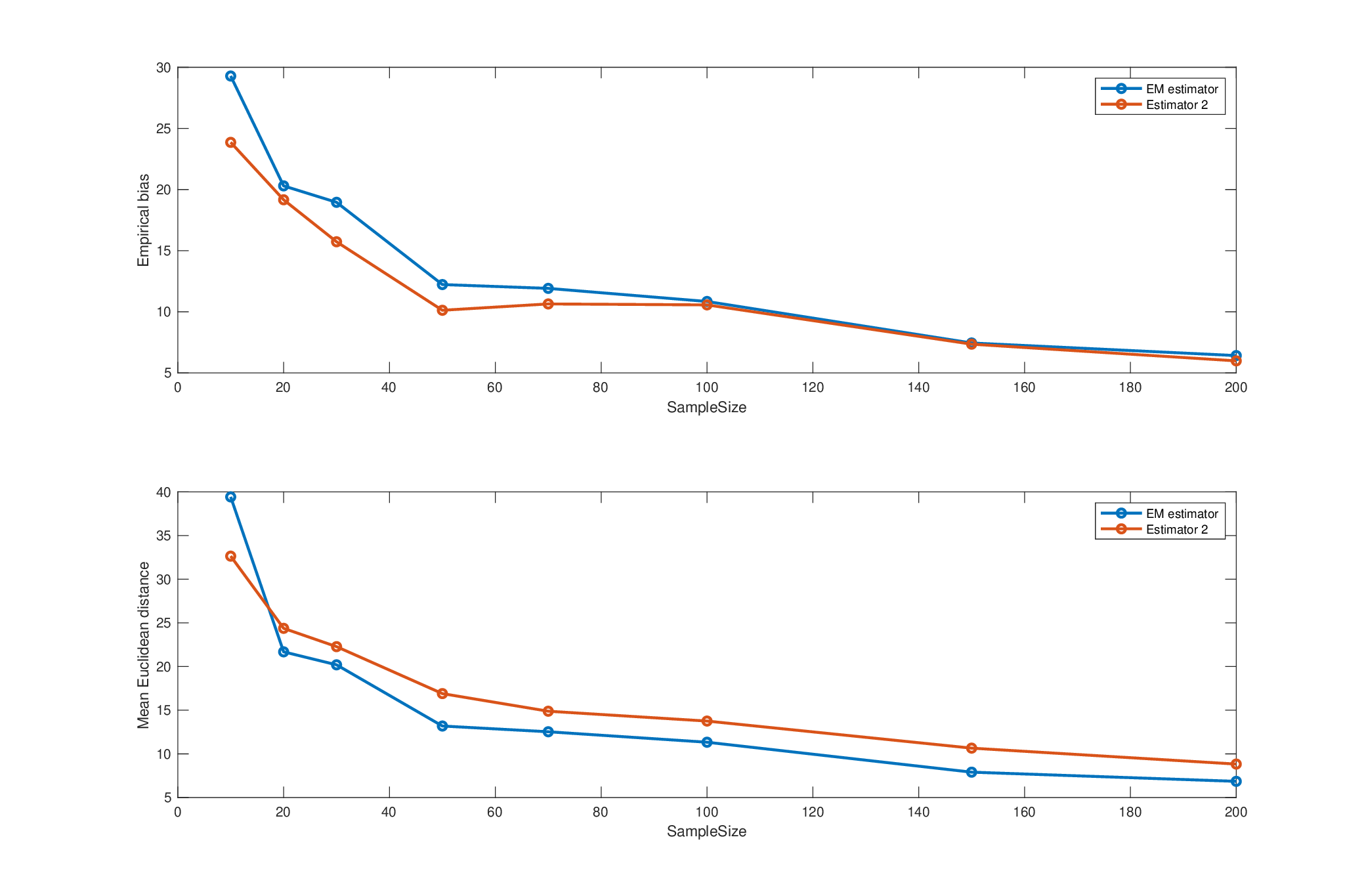}
		\caption{\textbf{Bias and mean Euclidean distance analysis of estimators $\hat{\bm{\Sigma}}$ and $\bm{\Sigma}^*$ over s simulation runs with respect to the sample size, for Case $6$.}}
		\label{fig:case6}
	\end{figure}

The simulation results lead to the following conclusions:
\begin{enumerate}[label=\textbullet]
	\item The figures $1$-$6$ illustrate the empirical bias and mean Euclidean distance of the EM estimator and the estimator $2$, denoted as $\bm{\Sigma}^*$ in \eqref{eq:25}. In all the cases, the empirical bias decreases as the sample size increases, for both estimators. Additionally, as the sample size increases, both estimator gives approximately the same bias.
	
	\item The mean Euclidean distance between the estimate and the parameter provides the accuracy of the estimator around the parameter. It is observed that the mean Euclidean distance decreases as the sample size increases in all the cases for both estimators. However, compared to estimator $\bm{\Sigma}^*$, the EM estimator demonstrates lower the mean Euclidean distance when the sample size increases. 
\end{enumerate}

From these results, we can conclude that the EM estimator is more consistent as compared to the estimator $\bm{\Sigma}^*$.

For the multivariate symmetric Laplace distribution $\mathcal{SL}_{p}(\bm{\Sigma})$, MLE of the parameter $\bm{\Sigma}$ exists if the sample size $N\ge p$. And it is observed that the maximum likelihood estimator of $\bm{\Sigma}$ uniquely exists for the same data sample, with different initial values in this algorithm.

	\subsection{The performance of the proposed estimators of \texorpdfstring{$\bm{\Sigma}_{1}$}{PDFstring} and \texorpdfstring{$\bm{\Sigma}_{2}$}{PDFstring} in  \texorpdfstring{$\mathcal{MSL}_{p,q}(\bm{\Sigma}_{1},\bm{\Sigma}_{2})$}{PDFstring} }
	
	In this section, the performance of proposed estimators of $\bm{\Sigma}_{1}$ and $\bm{\Sigma}_{2}$ in matrix variate symmetric Laplace distributions are shown using simulation. The performance of estimators $\hat{\bm{\Sigma}}_{1}, \hat{\bm{\Sigma}}_{2}$ is measured on following metric:  
	\begin{enumerate}
		\item Empirical bias :- $\| (\hat{\bm{\Sigma}}_{2} \otimes \hat{\bm{\Sigma}}_{1}) _{m}- \bm{\Sigma}_{2}\otimes \bm{\Sigma}_{1}\|_{2}$.
		\item Relative empirical bias :- \[\frac{\| (\hat{\bm{\Sigma}}_{2} \otimes \hat{\bm{\Sigma}}_{1}) _{m}- \bm{\Sigma}_{2}\otimes \bm{\Sigma}_{1}\|_{2}}{\| \bm{\Sigma}_{2} \otimes \bm{\Sigma}_{1} \|_{2}}.\]
		\item Mean Euclidean distance :- \[ \| \hat{\bm{\Sigma}}_{2} \otimes \hat{\bm{\Sigma}}_{1} - \bm{\Sigma}_{2} \otimes \bm{\Sigma}_{1}\|_{2,m}\]
		\item Relative mean Euclidean distance:- \[\frac{\| \hat{\bm{\Sigma}}_{2} \otimes \hat{\bm{\Sigma}}_{1} - \bm{\Sigma}_{2} \otimes \bm{\Sigma}_{1}\|_{2,m} }{\| \bm{\Sigma}_{2} \otimes \bm{\Sigma}_{1} \|_{2}}.\]
	\end{enumerate}
	$(\hat{\bm{\Sigma}}_{2} \otimes \hat{\bm{\Sigma}}_{1})_{m}$ denotes the empirical mean of estimates of $ \bm{\Sigma}_{2} \otimes \bm{\Sigma}_{1}$ over all simulations, and $\|.\|_{2,m}$ denotes the empirical mean of norms over all simulations.

	The simulations illustrate several key aspects of the estimators, including the convergence of the proposed algorithm, the asymptotic reduction of the empirical bias of $\hat{\bm{\Sigma}}_{2}\otimes \hat{\bm{\Sigma}}_{1}$ to zero and the mean Euclidean distance between the estimate and the actual parameter decreases, over time or in other words, the accuracy of the estimators increases. For simulation, four structures are considered for $\bm{\Sigma}_{1}$ and $\bm{\Sigma}_{2}$, given as Case 1-4. For all the cases, $p=5, q=3$ and the sample size $N$ as $5, 10, 15, 20, 30, 50$ and $100$. The number of simulation runs, $s$ is $200$ for all sample sizes.

	The four structures considered for $\bm{\Sigma}_{1}$ and $\bm{\Sigma}_{2}$ are:
	
	\renewcommand{\labelenumii}{\arabic{enumi}.\arabic{enumii}}
	
	\begin{enumerate}[label=\textbf{Case \arabic{enumi}.}]
		
		\item $\bm{\Sigma}_{1}= \setlength{\arraycolsep}{3pt} \begin{bmatrix}  1&0&0&0&0\\0&0.5&0&0&0\\ 0&0&2&0&0\\0&0&0&3&0\\0&0&0&0&0.65\end{bmatrix}$  and  $\bm{\Sigma}_{2}=  \setlength{\arraycolsep}{3pt} \begin{bmatrix} 3&0&0\\0&2&0\\0&0&1\end{bmatrix}, $\\
		
		$ (\| \bm{\Sigma}_{2} \otimes \bm{\Sigma}_{1} \| _{2} =14.3323 )$. \\
		
		\item $\bm{\Sigma}_{1}= \setlength{\arraycolsep}{3pt} \begin{bmatrix}1&0&0&0&0\\0&0.5&0&0&0\\ 0&0&2&0&0\\0&0&0&3&0\\0&0&0&0&0.65	\end{bmatrix}$ and 
		$\bm{\Sigma}_{2}= \setlength{\arraycolsep}{3pt} \begin{bmatrix} 3&1.5&1\\1.5&2&0\\1&0&1  \end{bmatrix},$	\\
		
		($\| \bm{\Sigma}_{2} \otimes \bm{\Sigma}_{1} \| _{2} =17.3432$). \\
		
		\item  $\bm{\Sigma}_{1}= \setlength{\arraycolsep}{3pt} \begin{bmatrix}	5&3&2.5&2&1.5\\3&4&2&1.5&1\\2.5&2&3&1&0.5\\2&1.5&1&2&0.2\\1.5&1&0.5&0.2&1
		\end{bmatrix}$ and
		$\bm{\Sigma}_{2}=  \setlength{\arraycolsep}{3pt} \begin{bmatrix} 3&0&0\\0&2&0\\0&0&1 \end{bmatrix},$	\\
		
		($\| \bm{\Sigma}_{2} \otimes \bm{\Sigma}_{1} \| _{2} =40.1388$).\\
		
		\item $\bm{\Sigma}_{1}= \setlength{\arraycolsep}{3pt} \begin{bmatrix}
			5&3&2.5&2&1.5\\3&4&2&1.5&1\\2.5&2&3&1&0.5\\2&1.5&1&2&0.2\\1.5&1&0.5&0.2&1	\end{bmatrix}$ and
		$\bm{\Sigma}_{2}= \setlength{\arraycolsep}{3pt} \begin{bmatrix} 	4&1&2\\1&5&3\\2&3&6	  \end{bmatrix},$ \\
		
		($\| \bm{\Sigma}_{2} \otimes \bm{\Sigma}_{1} \| _{2} =109.9245$).
		
	\end{enumerate}

	\begin{enumerate}[label=\textbf{Case \arabic{enumi}.}]
		\item Both the matrices $\bm{\Sigma}_{1}$ and $\bm{\Sigma}_{2}$ are diagonal.
		\item $\bm{\Sigma}_{1}$ is a diagonal matrix, while $\bm{\Sigma}_{2}$ is a non-diagonal matrix with less zeros.
		\item $\bm{\Sigma}_{1}$ is full matrix or have all non zero entries, while $\bm{\Sigma}_{2}$ is diagonal.
		\item Both $\bm{\Sigma}_{1}$ and $\bm{\Sigma}_{2}$ are full matrices.
	\end{enumerate}
	
	Observations from $\mathcal{MSL}_{p,q}(\bm{\Sigma}_{1},\bm{\Sigma}_{2})$ are generated using the representation in the theorem \ref{theorem:3}.
	
	In all cases, the initial values $\hat{\bm{\Sigma}}_{1} ^{(0)}$, $\hat{\bm{\Sigma}}_{2} ^{(0)}$ are taken as
	
	\[\hat{\bm{\Sigma}}_{1} ^{(0)}= \frac{1}{qN} \sum_{i=1}^{N} \bm{X}_{i} \bm{X}_{i}^\top,\] 
	\[ \hat{\bm{\Sigma}}_{2} ^{(0)}=\frac{1}{pN} \sum_{i=1}^{N} \bm{X}_{i}^\top \bm{X}_{i},\]
	
	where $N$ is the number of sample observations. The initial estimates depend upon the samples, and $\epsilon=10^{-11}$.

	The simulation results lead to the following conclusions:
	
	\begin{figure}[hbt!]
		\centering
		\includegraphics[width=1\linewidth]{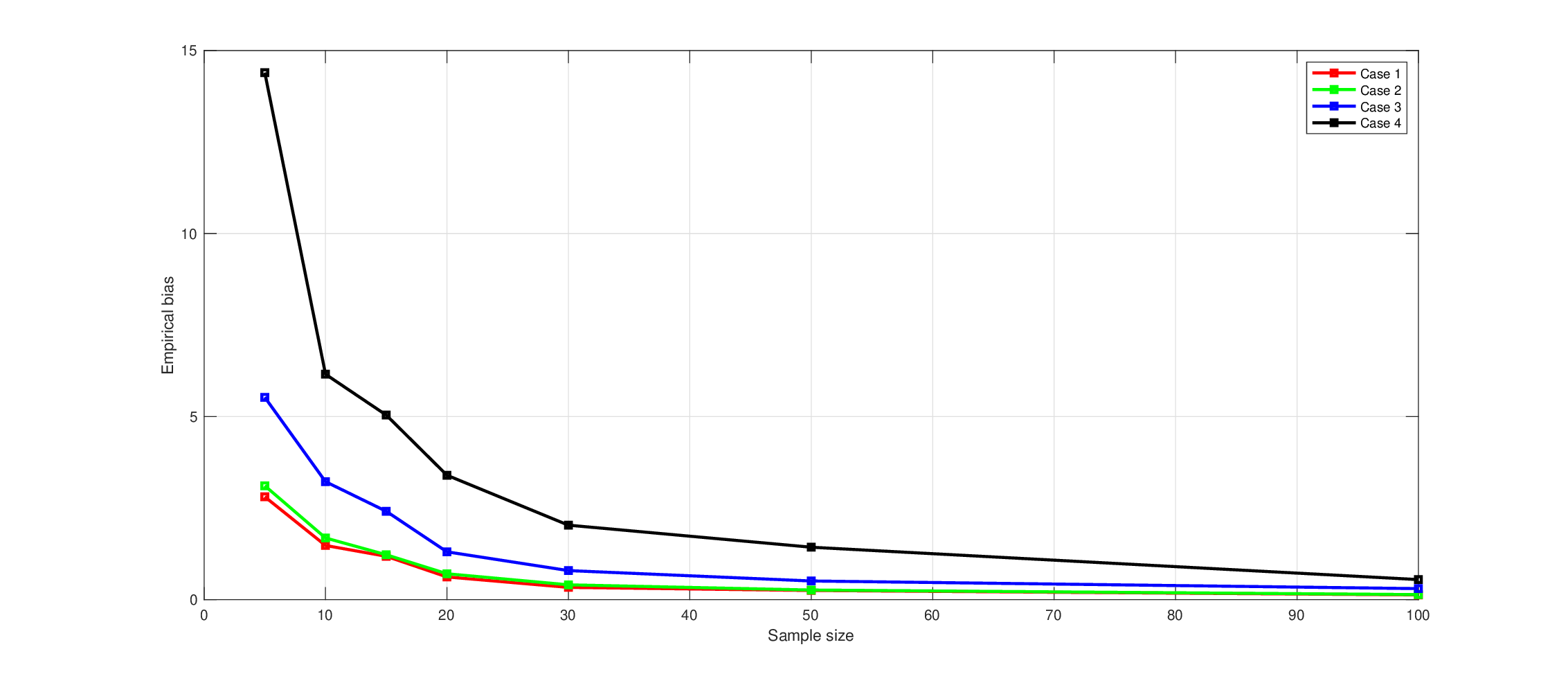}
		\caption{\textbf{\footnotesize Bias analysis of $\hat{\bm{\Sigma}}_{2} \otimes \hat{\bm{\Sigma}}_{1}$. Empirical bias is defined as $\| (\hat{\bm{\Sigma}}_{2} \otimes \hat{\bm{\Sigma}}_{1}) _{m}- \bm{\Sigma}_{2}\otimes \bm{\Sigma}_{1}\|_{2}$, where $(\hat{\bm{\Sigma}}_{2} \otimes \hat{\bm{\Sigma}}_{1}) _{m}$ denotes the empirical mean of $\hat{\bm{\Sigma}}_{2} \otimes \hat{\bm{\Sigma}}_{1} $ over $s$ simulation runs, with respect to the sample size for all four cases.}}
		\label{fig:figure1a}
		\centering
		\includegraphics[width=1\linewidth]{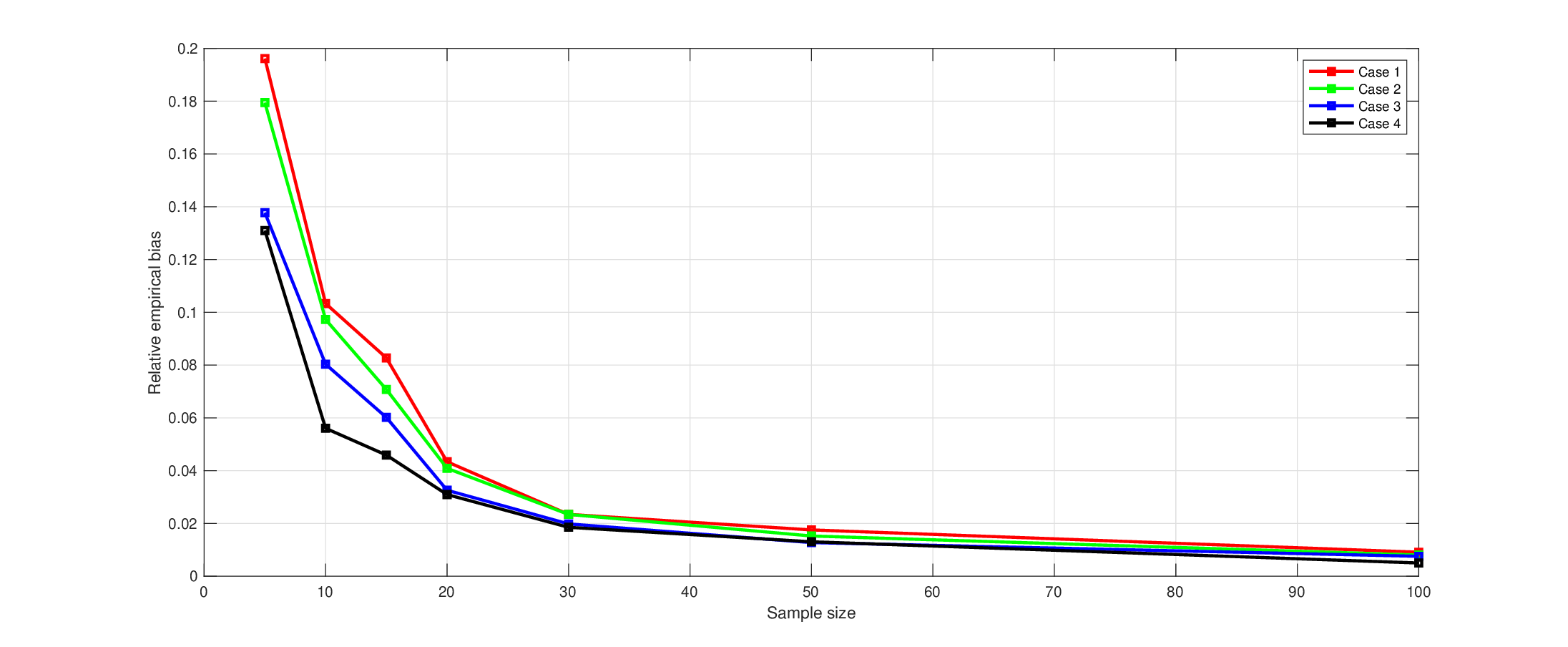}
		\caption{\footnotesize \textbf{Relative empirical bias, which is obtained by dividing the empirical bias with the Euclidean norm of $\bm{\Sigma}_{2} \otimes \bm{\Sigma}_{1}$, with respect to the sample size for all four cases.}}
		\label{fig:figure1b}
	\end{figure}

	\begin{table}[hbt!]
		\caption{\label{table:1}\textbf{\small Mean number of iterations required to meet the stopping criterion of this algorithm, for all the Cases (Case 1-4), with $\epsilon=10^{-11}$. For all Cases $p=5,q=3$ (Cases 1-4 are as given above in this section).}}
		\centering
		\begin{tabular}{|c|c|c|c|c|c|}
			\hline
			\textbf{N} & \textbf{s} & \textbf{Case 1}& \textbf{Case 2} & \textbf{Case 3}& \textbf{Case 4}\\
			\hline
			5 &200&103&101&110&120\\
			10&200&111&107&118&127\\
			15&200&114&112&121&129\\
			20&200&116&114&124&131\\
			30&200&118&118&125&133\\
			50&200&121&120&129&136\\
			100&200&124&124&132&140\\  
			\hline    
		\end{tabular}	 
		\vspace{2mm}
	\end{table}

	\begin{table}
		\caption{\label{table:2} \textbf{\small Mean Euclidean distance $\| \hat{\bm{\Sigma}}_{2} \otimes \hat{\bm{\Sigma}}_{1} - \bm{\Sigma}_{2} \otimes \bm{\Sigma}_{1}\|_{2,m}$  between estimate $\hat{\bm{\Sigma}}_{2} \otimes \hat{\bm{\Sigma}}_{1}$ and the parameter $\bm{\Sigma}_{2} \otimes \bm{\Sigma}_{1}$, where $m$ refers the mean of the Euclidean distance over s simulation runs.}}
		\centering	
		\begin{tabular}{|c|c|c|c|c|c|}
			\hline
			\textbf{N} & \textbf{s} & \textbf{Case 1}& \textbf{Case 2} & \textbf{Case 3}& \textbf{Case 4}\\
			\hline
			5&200&15.3985&16.2440&34.6415&87.7367\\
			10&200&8.9114&10.6120&23.3722&53.8132\\
			15&200&7.2226&8.1232&17.4023&44.6107\\
			20&200&6.0524&7.2905&13.8990&37.2770\\
			30&200&4.6809&5.4556&11.3208&30.3724\\
			50&200&3.7099&4.0746&8.5270&22.7902\\
			100&200&2.5898&2.8977&6.5092&16.0387\\
			\hline
		\end{tabular}	
	\end{table}
	
	\begin{table}[hbt!]
		\caption{\label{table:3}\textbf{\small Relative mean Euclidean distance between estimate $\hat{\bm{\Sigma}}_{2} \otimes \hat{\bm{\Sigma}}_{1}$  and the parameter  $\bm{\Sigma}_{2} \otimes \bm{\Sigma}_{1}$, (Relative mean Euclidean distance is given as dividing the mean Euclidean distance by norm of the parameter), i.e. $\frac{\| \hat{\bm{\Sigma}}_{2} \otimes \hat{\bm{\Sigma}}_{1} - \bm{\Sigma}_{2} \otimes \bm{\Sigma}_{1}\|_{2,m}}{\| \bm{\Sigma}_{2} \otimes \bm{\Sigma}_{1} \|_{2} }$.}}
		\centering
		\begin{tabular}{|c|c|c|c|c|c|}
			\hline
			\textbf{N} & \textbf{s} & \textbf{Case 1}& \textbf{Case 2} & \textbf{Case 3}& \textbf{Case 4}\\
			\hline
			5&200&1.0744&0.9366&0.8630&0.7982\\
			10&200&0.6218&0.6119&0.5823&0.4895\\
			15&200&0.5039&0.4684&0.4336&0.4095\\
			20&200&0.4223&0.4204&0.3463&0.3391\\
			30&200&0.3266&0.3146&0.2820&0.2763\\
			50&200&0.2588&0.2349&0.2124&0.2073\\
			100&200&0.1807&0.1671&0.1622&0.1459\\
			\hline
		\end{tabular}
	\end{table} 

	\begin{figure}[hbt!]
		\centering 
		\includegraphics[width=1\linewidth]{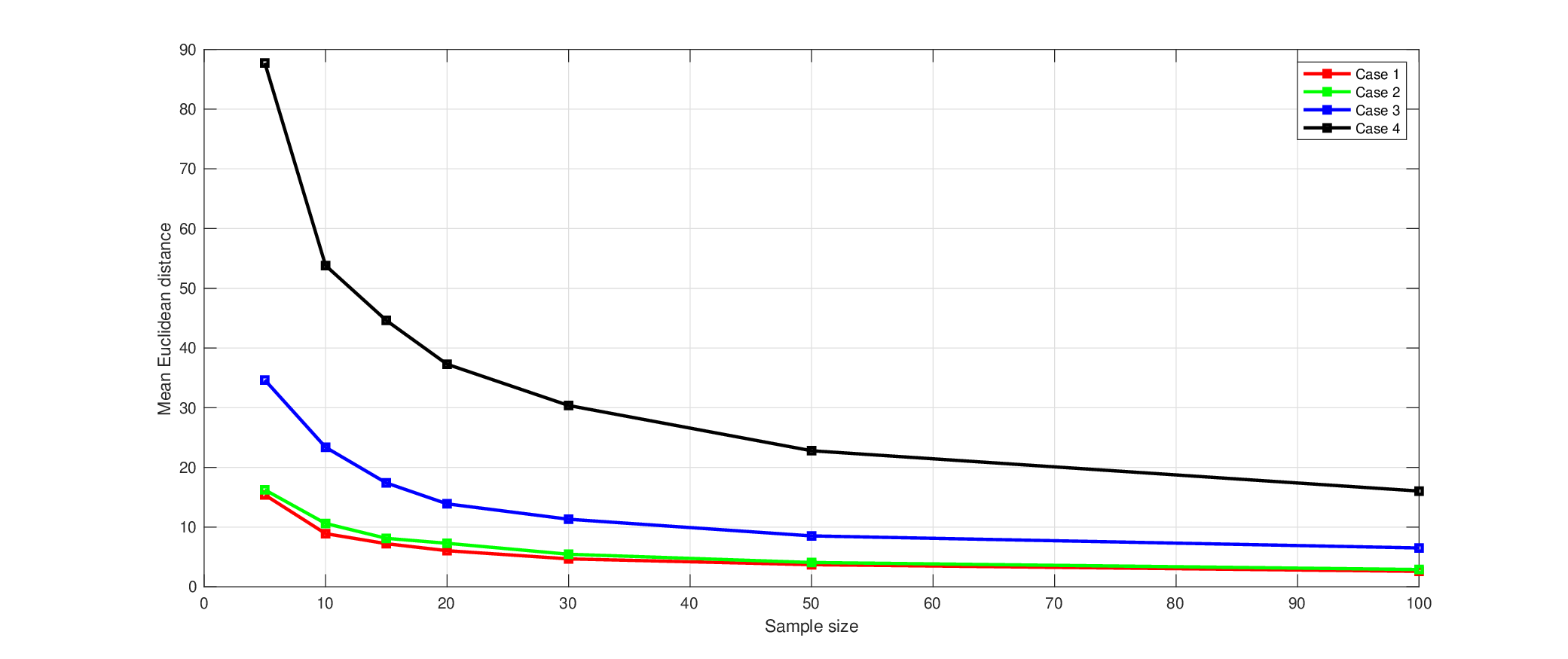}
		\caption{\small \textbf{Mean Euclidean distance between the estimate and the parameter, with respect to the sample size for all four cases.}}
		\label{fig:figure2a}
		
		\centering
		\includegraphics[width=1\linewidth]{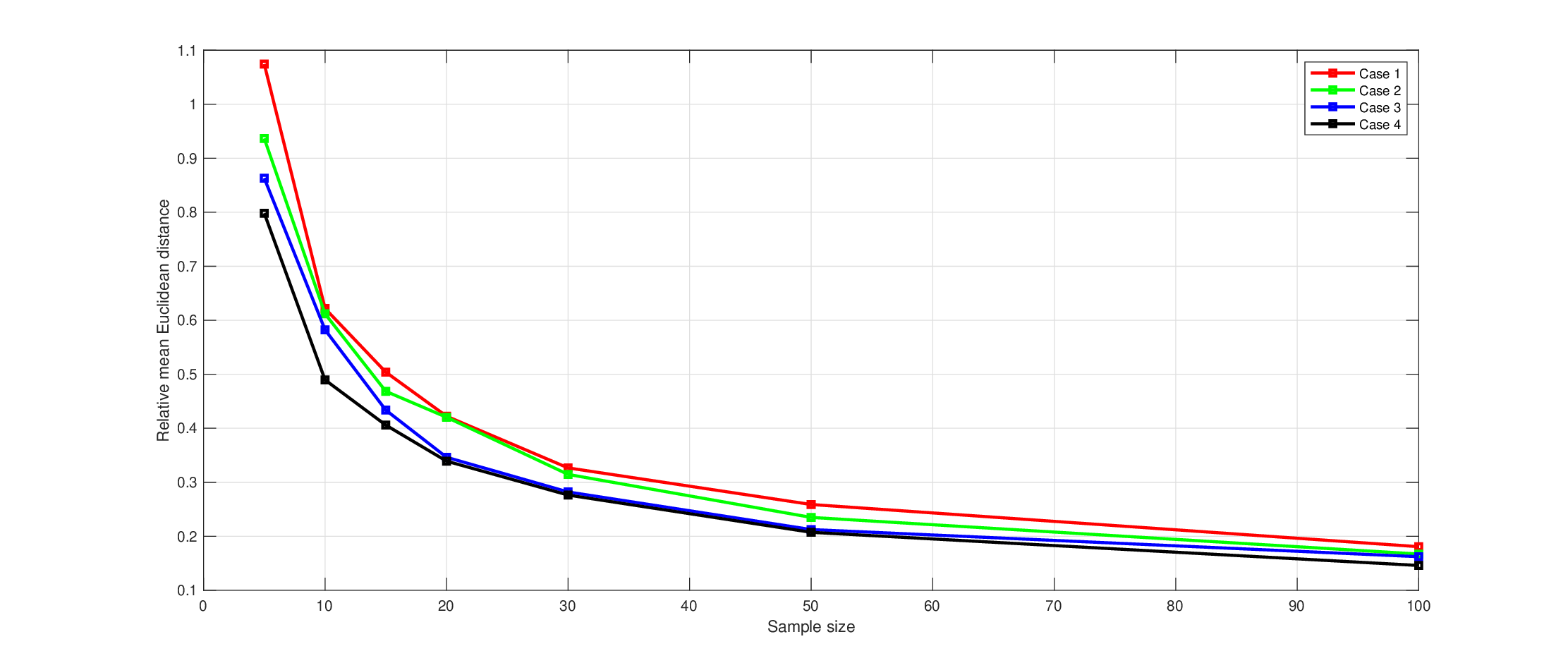}
		\caption{\small \textbf{Relative mean Euclidean distance between the estimate and the parameter, with respect to the sample size for all four cases.}}
		\label{fig:figure2b}
	\end{figure}

	\begin{enumerate}[label=\textbullet]
		\item The initial estimates of $\bm{\Sigma}_{1}$ and $\bm{\Sigma}_{2}$  are taken as $\hat{\bm{\Sigma}}_{1} ^{(0)}$ and $\hat{\bm{\Sigma}}_{2} ^{(0)}$, depending on the samples. Using the initial values $\hat{\bm{\Sigma}}_{1} ^{(0)}$ and $\hat{\bm{\Sigma}}_{2} ^{(0)}$ other than these may affect the number of iterations to meet the stopping criterion. There is a slight difference in the mean number of iterations between case 1 and case 2. 
		
		\item The figure \ref{fig:figure1a} shows that the empirical bias of the estimator, defined as the Euclidean distance between the empirical mean of estimates over $s$ simulation runs and the parameter, decreases as sample size $N$ increases for all the four cases.
		
		\item Relative empirical bias is obtained by dividing the empirical bias to the Euclidean norm of the parameter.  Relative empirical bias is also asymptotically decreases to zero for all the cases, as shown in the figure \ref{fig:figure1b}. If we compare all the cases, in the figure \ref{fig:figure1a}, Case $4$ has the largest bias and Case $1$ has the lowest, but in the figure \ref{fig:figure1b}, Case $4$ has lowest relative bias and Case $1$ has largest relative bias. This happens becuase of the scale or unit of the true parameters, which is measured by their norms. Case $4$ has large scale and Case $1$ has small scale. If we see the empirical bias, it says that Case $1$ has the better estimator as compared to Case $4$, but Case $4$ has the better accurate estimator as compared to Case $1$.

		\item  The empirical bias defined above is a global measure, or it gives the overall error of the estimate, not the individual of $\bm{\Sigma_{1}}$ and $\bm{\Sigma_{2}}$. Since it is measured by the norm, it always takes positive values. Therefore, it does not reveal the individual estimates, whether they are overestimated or underestimated.

		\item The mean Euclidean distance between the estimate and the parameter provides the overall accuracy of the estimator around the parameter. From the table \ref{table:2} and the figure \ref{fig:figure2a}, it is observed that the mean Euclidean distance decreases as the sample size increases for all the cases. So, $ \hat{\bm{\Sigma}}_{2} \otimes \hat{\bm{\Sigma}}_{1}$ can be considered as a consistent estimator of $\bm{\Sigma}_{2} \otimes \bm{\Sigma}_{1}$. 
		
		\item The relative mean Euclidean distance is obtained by dividing the mean Euclidean distance by the norm of the parameter. In the table \ref{table:3} and the figure \ref{fig:figure2b}, it is observed that the relative mean Euclidean distance is approaching zero as the sample size increases. The table \ref{table:2} shows that Case $4$ has the largest value of the mean Euclidean distance, and Case $1$ has the lowest value of the mean Euclidean distance. After the relative result, the table \ref{table:3} shows the lowest value of the mean Euclidean distance for Case $4$ and the largest value of the mean Euclidean distance for Case $1$, which emphasizes the use of relative the mean Euclidean distance rather than simple the mean Euclidean distance for the measurement of the performance of estimators.
	\end{enumerate}
	Thus, the proposed algorithm for the matrix variate symmetric Laplace distribution estimates $\bm{\Sigma}_{2} \otimes \bm{\Sigma}_{1}$ for all four structures considered in nominal iterations.
	
	The proposed algorithm for the matrix variate symmetric Laplace distribution can be applied to estimate the parameter $\bm{\Sigma}$ of multivariate symmetric Laplace distribution, if $\bm{\Sigma}$ can be decomposed into $\bm{\Sigma}_{2} \otimes \bm{\Sigma}_{1}$, where, $\bm{\Sigma}_{1}, \bm{\Sigma}_{2}$ are positive definite matrices, even when the sample size is small.
	 \begin{note}
		From the characteristic function \ref{thm:4}, the parameters $\bm{\Sigma}_{1}$ and $\bm{\Sigma}_{2}$ are defined up to a positive multiplicative constant, that is, 
		$\left(\bm{\Sigma}_{1}, \bm{\Sigma}_{2}\right)$ and $\left(a \bm{\Sigma}_{1}, (1/a) \bm{\Sigma}_{2}\right)$ with $a>0$ follows the same distribution. Thus, we may get estimates $\left(\hat{\bm{\Sigma}}_{1}, \hat{\bm{\Sigma}}_{2}\right)$ which are actually estimating $\left(a \bm{\Sigma}_{1}, (1/a) \bm{\Sigma}_{2} \right)$ with $a>0$ rather than $\left(\bm{\Sigma}_{1},\bm{\Sigma}_{2}\right)$, but the Kronecker product $\hat{\bm{\Sigma}}_{2}\otimes \hat{\bm{\Sigma}}_{1}$ is estimating $\bm{\Sigma}_{2}\otimes \bm{\Sigma}_{1}$. 
	\end{note} 
	
	\section{Conclusion}
	In this paper, the maximum likelihood estimators of the parameters of both multivariate and matrix variate symmetric Laplace distributions are proposed using the EM algorithm. The existence conditions of the proposed estimators of both multivariate and matrix variate symmetric Laplace distribution are also given. The performance of the proposed EM estimator is compared with another estimator of the multivariate symmetric Laplace distribution by evaluating bias and mean Euclidean distance of these estimators. The results indicate that the EM estimator is more consistent. Furthermore, the performance of proposed estimators of the matrix variate symmetric Laplace distribution is evaluated using two metrics, the empirical bias and mean Euclidean distance of the Kronecker product of estimators. This evaluation is conducted across four different structures using simulated data sets. This simulation study reveals that the empirical bias decreases in all the cases as the sample size increases. Similarly, the mean Euclidean distance decreases with a large sample size, indicating that the estimator can be considered consistent. However, this simulation study does not address individual estimators of $\bm{\Sigma_{1}}$ and $\bm{\Sigma_{2}}$. Based on this simulation study, the matrix variate symmetric Laplace distribution can be an alternative to the multivariate symmetric Laplace distribution when a small number of sample observations are available. This is especially beneficial when the scale parameter of the multivariate symmetric Laplace distribution can be decomposed as the Kronecker product of two positive definite matrices.

 \noindent \textbf{Acknowledgements:} The first author would like to thank the University Grants Commission, India, for providing financial support.


\begin{thebibliography}{99}
			
            \bibitem{A}O. Arslan, An alternative multivariate skew Laplace distribution: properties and estimation, \emph{Statistical Papers}, \textbf{51}(4), (2010) 865-887.
			\bibitem{B}F. Bowman, \emph{Introduction to Bessel Functions}, Courier Corporation, (2012).
             \bibitem{BA} Y. M. Bulut and O. Arslan, Matrix variate skew laplace distribution, \emph{Sigma Journal of Engineering and Natural Sciences}, \textbf{42}(3), (2024) 854-861.
			\bibitem{CHS} S. Cambanis, S. Huang, and G. Simons, On the theory of elliptically contoured distributions, \emph{Journal of Multivariate Analysis}, \textbf{11}(3), (1981) 368-385.
			\bibitem{Du}P. Dutilleul, The MLE algorithm for the matrix normal distribution, \emph{Journal of statistical computation and simulation}, \textbf{64}(2), (1999) 105-123.
			\bibitem{Dw}P. S. Dwyer, Some applications of matrix derivatives in multivariate analysis, \emph{Journal of the American Statistical Association}, \textbf{62}(318), (1967) 607-625.
			\bibitem{DLR}A. P. Dempster, N. M. Laird, and D. B. Rubin, Maximum likelihood from incomplete data via the EM algorithm, \emph{Journal of the Royal Statistical Society: series B (methodological)}, \textbf{39}(1), (1977) 1-22.
			\bibitem{EKL}T. Eltoft, T. Kim, and T. W. Lee, On the multivariate Laplace distribution, \emph{IEEE Signal Processing Letters}, \textbf{13}(5), (2006) 300-303.
			\bibitem{FM}K. Fragiadakis and S. G. Meintanis, Goodness-of-fit tests for multivariate Laplace distributions, \emph{Mathematical and Computer Modelling}, \textbf{53}(5-6), (2011) 769-779.
			\bibitem{Gr}F. A. Graybill, \emph{Matrices with Applications in Statistics}, Second edition, Wadsworth, Belmont (1983).
			\bibitem{GN}A. K. Gupta and D. K. Nagar,  \emph{Matrix Variate Distributions}, Chapman and Hall/CRC (2018).
			\bibitem{KKP}S. Kotz, T. Kozubowski, and K. Podgórski, \emph{The Laplace distribution and Generalizations: a Revisit with Applications to Communications, Economics, Engineering, and Finance}, Springer Science and Business Media, (2001).
			\bibitem{KMP}T.Kozubowski, S. Mazur, and K. Podgórski, Matrix variate generalized asymmetric Laplace distributions, \emph{Theory of Probability and Mathematical Statistics} \textbf{109}, (2023) 55-80.
			\bibitem{KP}T. J. Kozubowski and K. Podgórski, Asymmetric Laplace laws and modeling financial data, \emph{Mathematical and Computer Modelling}, \textbf{34}(9-11), (2001) 1003-1021.
           \bibitem{KPo} T. J. Kozubowski and K. Podgórski, A multivariate and asymmetric generalization of Laplace distribution. \emph{Computational Statistics}, \textbf{15}(4), (2000) 531-540.
			\bibitem{KPR}T. J. Kozubowski, K. Podgórski, and  I. Rychlik, Multivariate generalized Laplace distribution and related random fields, \emph{Journal of Multivariate Analysis}, \textbf{113},  (2013) 59-72.
			\bibitem{KS}T. Kollo and M. S. Srivastava, Estimation and testing of parameters in multivariate Laplace distribution, \emph{Communications in Statistics-Theory and Methods}, \textbf{33}(10), (2005) 2363-2387.
			\bibitem{MK}GJ. McLachlan and T. Krishnan, \emph{The EM Algorithm and Extensions}, John Wiley and Sons, (2007).
			\bibitem{RC}H. Roger and RJ. Charles, \emph{Topics in matrix analysis}, Cambridge University Press, Cambridge, England, (1994).
			\bibitem{OLBC} F. W. Olver, D. W. Lozier, R. F. Boisvert, and  C. W. Clark, \emph{NIST Handbook of Mathematical Functions}, Cambridge University Press (2010).
			\bibitem{VA}T. Varga,  \emph{Matrix variate elliptically contoured distributions: stochastic representation and inference}, (Doctoral dissertation, Bowling Green State University) (1990).
			\bibitem{Vi}H. Visk, On the parameter estimation of the asymmetric multivariate Laplace distribution, \emph{Communications in Statistics-Theory and Methods}, \textbf{38}(4), (2009) 461-470.
			\bibitem{W}G. N. Watson, \emph{A Treatise on the Theory of Bessel Functions (Vol. 2)}, The University Press, (1922).
			\bibitem{Y}Y. Yurchenko, Matrix variate and tensor variate Laplace distributions, arXiv preprint arXiv:2104.05669 (2021).
			
		\end{thebibliography}
\end{document}